\documentclass[a4paper,12pt]{article}
\usepackage{tikz}
\usetikzlibrary{arrows}
\usetikzlibrary{calc}
\usetikzlibrary{positioning}
\usetikzlibrary{cd}
\usepackage{authblk}
\usepackage{graphicx,amsmath,amssymb,amsthm,amscd,mathrsfs,mathtools}
\theoremstyle{plain}
\newtheorem{introthm}{Theorem}

\usepackage[margin=2.5cm]{geometry}
\numberwithin{equation}{section}

\newtheorem{thm}{Theorem}[section]
\newtheorem{prop}[thm]{Proposition}
\newtheorem{lem}[thm]{Lemma}
\newtheorem{cor}[thm]{Corollary}

\theoremstyle{definition}
\newtheorem{df}[thm]{Definition}

\newtheorem{ass}[thm]{Assumption}
\theoremstyle{remark}
\newtheorem{rem}[thm]{Remark}

\title{Middle convolution for affine hyperplane arrangements and Aomoto realizability}
\date{}

\author[1]{Shunya Adachi\thanks{The author was supported by JSPS KAKENHI Grant Number JP24K22826.}} 
\author[2]{Kazuki Hiroe\thanks{The author was supported by JSPS KAKENHI Grant Number JP25K07043.}}
\affil[1]{\small{Cooperative Faculty of Education, Utsunomiya University, 
  350 Mine-machi, Utsunomiya-shi, Tochigi, 321-8505 JAPAN 
 email: {\tt sadachi@a.utsunomiya-u.ac.jp}}}
\affil[2]{Department of Mathematics and Informatics, Chiba University,
1-33 Yayoi-cho, Inage-ku, Chiba-shi, Chiba, 263-8522 JAPAN 
email: {\tt kazuki@math.s.chiba-u.ac.jp}}

\begin{document}
\maketitle
\begin{abstract}
We construct middle convolution for local systems on complements
of affine hyperplane arrangements.
Our aim is to study Aomoto realizability: whether a local system
can be realized by a flat logarithmic connection on a trivial
bundle with constant residue matrices.

For a chosen direction, we minimally enlarge the arrangement
so that projection becomes a bundle of punctured affine lines.
On the associated relative stratification, we extend Katz's
middle convolution to constructible perverse sheaves and prove
compatibility with its action on the fibers.
It induces middle convolution on local systems and satisfies
composition and inversion laws under fiberwise hypotheses.

Using logarithmic de Rham cohomology, we compare this construction
with Haraoka's algebraic middle convolution of logarithmic
Pfaffian systems under explicit spectral conditions.
The comparison requires no irreducibility assumption and allows
integral differences between nonintegral residue eigenvalues,
extending the scope of the classical comparison theorems.
We also prove flatness of the algebraic convolution
for every parameter.

Consequently, middle convolution preserves Aomoto realizability
on the enlarged arrangement.
For inverse middle convolution, we give sufficient conditions
for removing the additional poles and recovering an Aomoto
realization of the original local system on the initial arrangement.
\end{abstract}

\section*{Introduction}

Let $\mathcal A$ be a finite affine hyperplane arrangement in
$V=\mathbb C^\ell$, with complement $M(\mathcal A)$, and choose
an affine linear defining equation $f_H$ for each $H\in\mathcal A$.
We study the following realization problem:
\begin{quote}
\emph{Which local systems on $M(\mathcal A)$ are sheaves of
horizontal sections of flat logarithmic connections
\[
    \nabla_A=d-\Omega_A,\qquad
    \Omega_A=\sum_{H\in\mathcal A}A_H\,d\log f_H,
    \qquad A_H\in\operatorname{End}(E),
\]
with constant coefficient matrices on the trivial bundle
$M(\mathcal A)\times E$, for some finite-dimensional complex
vector space $E$?}
\end{quote}
We call such a presentation an \emph{Aomoto realization} and the
local system \emph{Aomoto-realizable with respect to $\mathcal A$}.
Write $\operatorname{Pf}(\log\mathcal A)$ for the category of these
connections, with constant intertwiners as morphisms, and
$\operatorname{DR}_{\mathrm{Pf}}$ for the horizontal section functor.

The requirement that the connection have this specific form
makes the realization problem nontrivial.
The Riemann--Hilbert correspondence supplies a regular singular
connection realizing a given local system \cite{DelEDR}, but does
not by itself provide a realization on a trivial bundle with
constant coefficient matrices in the specified logarithmic forms.
Here the arrangement is fixed, and the matrices must satisfy
the integrability relations and realize the prescribed monodromy.

Aomoto studied this restricted Riemann--Hilbert problem and
proved realizability for unipotent local systems on hyperplane
arrangement complements \cite{Aom78}.
Hain subsequently gave a structural characterization of realizable
monodromy representations in terms of continuous representations
of a topological algebra associated with the fundamental group
\cite[Theorem~6.1]{Hain86}.
His criterion recovers Aomoto's unipotent realizability theorem
\cite[Theorem~6.6]{Hain86}.
For prescribed monodromy, however, establishing the required
continuous factorization remains a substantive part of the
realization problem.

The unipotent case does not settle realizability for higher-rank
irreducible local systems.
Explicit hypergeometric constructions provide further families
\cite{Har1,Har2}, but additional criteria or constructions are
needed to address realizability for a general local system
on a prescribed arrangement.

Aomoto realizations also facilitate cohomology computations.
Let $\mathsf A^\bullet(\mathcal A)$ be the algebra generated by
the forms $d\log f_H$.
Under their global nonresonance condition,
Esnault--Schechtman--Viehweg identify local system cohomology
with that of the finite-dimensional complex
\[
    \bigl(\mathsf A^\bullet(\mathcal A)\otimes E,
          -\Omega_A\wedge\bigr)
\]
\cite{ESV,ESVerr}.
Thus an Aomoto realization satisfying this condition provides
a concrete model for cohomology.

We study realizability through middle convolution.
Our comparison between topological and algebraic middle
convolution allows an initial Aomoto realization to be propagated
along successive convolutions, provided that the spectral
conditions at each step hold.
The construction may enlarge the arrangement.
For inverse convolution, we give sufficient conditions for
removing the additional poles and recovering an Aomoto realization
of the original local system on the original arrangement.
These results provide a systematic way to construct further
realizable local systems from known ones.
The cohomological method underlying the comparison also extends
the range of the classical comparison theorems.

\paragraph{Relative middle convolution.}

Fix a line $Y\subset V$ through the origin, and let
$p_Y:V\to V/Y$ be the quotient map.
The restriction of $p_Y$ to $M(\mathcal A)$ need not be a locally
trivial bundle of punctured affine lines, because intersections
with transverse hyperplanes may collide.
The \emph{$Y$-closure} $\mathcal B=\overline{\mathcal A}^{\,Y}$
is the minimal enlargement removing this obstruction:
for distinct transverse hyperplanes $H,H'$ with nonempty
intersection, it adjoins $(H\cap H')+Y$ when necessary.
Let $\mathcal B_Y$ be the subarrangement parallel to $Y$ and set
\[
    S:=M(p\mathcal B_Y),\qquad M:=M(\mathcal B),\qquad
    I:=\mathcal B\setminus\mathcal B_Y
       =\mathcal A\setminus\mathcal A_Y.
\]
Here $p\mathcal B_Y$ is the induced arrangement in $V/Y$, and
$\mathcal A_Y$ is defined similarly.
Then $\pi_Y:M\to S$ is a locally trivial bundle of punctured
affine lines.
Denote the inclusion by
$\iota_Y:M\hookrightarrow M(\mathcal A)$.

Choose a coordinate $y$ along $Y$, so that the transverse
hyperplanes are graphs $y=a_H(\mathbf x)$.
Their restrictions $\Gamma_H$ over $S$ are pairwise disjoint.
They define the relative stratification
$\mathscr S=\{M\}\cup\{\Gamma_H:H\in I\}$ of
$\mathcal X=S\times\mathbb A^1$; write
$j:M\hookrightarrow\mathcal X$.
For the fiber inclusion
$i_{\mathbf x}:\{\mathbf x\}\times\mathbb A^1
\hookrightarrow\mathcal X$, put
$P_{\mathbf x}:=i_{\mathbf x}^*P[-\dim S]$.
Let $\mathcal P^Y(\mathcal X)$ be the full subcategory of
$\operatorname{Perv}_{\mathscr S}(\mathcal X)$ for which every
$P_{\mathbf x}$ has Katz's Property $\wp$: it has no nonzero
constant perverse subobject or quotient.
We recall this condition in
Section~\ref{subsec:relative-perverse-MC}.

\begin{introthm}[Relative perverse middle convolution]
\label{thm:intro-relative-MC}
For every character $\chi:\mathbb Z\to\mathbb C^\times$, there is
a functor
\[
    \mathsf{MC}_\chi^Y:
    \operatorname{Perv}_{\mathscr S}(\mathcal X)
       \longrightarrow\operatorname{Perv}_{\mathscr S}(\mathcal X),
\]
compatible with Katz's middle convolution on each fiber, with
$\mathsf{MC}_1^Y\cong\operatorname{Id}$.
It preserves $\mathcal P^Y(\mathcal X)$, and on this subcategory
there are natural isomorphisms for all characters $\chi,\varphi$,
\[
    \mathsf{MC}_\varphi^Y\mathsf{MC}_\chi^Y
       \cong\mathsf{MC}_{\chi\varphi}^Y.
\]
In particular, $\mathsf{MC}_\chi^Y$ is an autoequivalence with
inverse $\mathsf{MC}_{\chi^{-1}}^Y$.
\end{introthm}

For $\chi\neq1$, passage to middle extensions and restriction
to $M$ induce middle convolution on local systems:
\[
    \operatorname{MC}_\chi^Y(\mathcal L)
       :=j^*\mathsf{MC}_\chi^Y
          \bigl(j_{!*}(\mathcal L[\ell])\bigr)[-\ell].
\]
For arbitrary $\mathcal A$, we set
$\operatorname{MC}_\chi^{Y,\mathcal A}(\mathcal L)
 :=\operatorname{MC}_\chi^{Y,\mathcal B}(\iota_Y^*\mathcal L)$.
The hypotheses under which the perverse composition law descends
to local systems are stated in Theorem~\ref{thm:composition}.

\paragraph{Comparison with algebraic middle convolution.}

Dettweiler--Reiter \cite{DR07} established the classical comparison
for Fuchsian systems on the projective line.
Haraoka extended algebraic middle convolution to logarithmic
Pfaffian systems on affine hyperplane complements \cite{Har1}
and studied the corresponding comparison for braid arrangements
\cite{Har2}.
In our notation, the algebraic construction gives
\[
    \operatorname{mc}_\lambda^Y:
    \operatorname{Pf}(\log\mathcal A)
       \longrightarrow\operatorname{Pf}(\log\mathcal B).
\]
Thus both constructions naturally involve the same $Y$-closure.

For $\nabla_A\in\operatorname{Pf}(\log\mathcal A)$, put
$A_I:=\sum_{H\in I}A_H$.
For $\lambda\notin\mathbb Z$, consider the condition
\[
\tag{$\mathrm{SC}_\lambda$}
\begin{aligned}
    \operatorname{Spec}(A_H)
       &\cap(\mathbb Z\setminus\{0\})=\varnothing
       &&(H\in I),\\
    \operatorname{Spec}(A_I+\lambda\operatorname{Id}_E)
       &\cap(\mathbb Z\setminus\{0\})=\varnothing.
\end{aligned}
\]
For other input connections, the same notation refers to
their own residues and underlying vector spaces.

\begin{introthm}[De Rham compatibility]
\label{thm:introMCcomp}
Suppose that $\lambda\notin\mathbb Z$ and
$(\mathrm{SC}_\lambda)$ holds.
For $\chi(1)=e^{2\pi i\lambda}$, there is a natural isomorphism
of local systems on $M(\mathcal B)$,
\[
    \operatorname{MC}_\chi^{Y,\mathcal A}
       \bigl(\operatorname{DR}_{\mathrm{Pf}}(\nabla_A)\bigr)
       \cong
    \operatorname{DR}_{\mathrm{Pf}}
       \bigl(\operatorname{mc}_\lambda^Y(\nabla_A)\bigr).
\]
\end{introthm}

Even in the classical settings, this comparison requires fewer
auxiliary hypotheses.
It requires neither irreducibility nor the auxiliary monodromy
hypotheses in \cite[Remark~4.8(a)]{DR07}.
It also allows integral differences between nonintegral
eigenvalues, which are excluded by the nonresonance assumptions
in \cite[Theorem~5.5]{Har2}.
The condition $(\mathrm{SC}_\lambda)$ concerns only the residues
relevant to the chosen fiber direction, imposes no restriction
on parallel residues, and is distinct from the
Esnault--Schechtman--Viehweg condition.

This wider scope comes from a cohomological treatment of the
period comparison.
Logarithmic de Rham cohomology proves that the full period map
is an isomorphism.
Local calculations then identify the boundary cycle subspace
with the subspace defining the algebraic middle quotient,
compatibly with variation in the transverse parameters.
These steps supply the completeness and quotient compatibility
needed in the earlier integral arguments; the points requiring
further justification there are discussed in
Remark~\ref{rem:comparison-normalization}.
Together with Haraoka's period equations, completeness also
gives a proof of flatness.
Polynomial dependence of the full convolution residues extends
flatness to every parameter and every flat input, as explained
in Section~\ref{subsec:convolution-flatness}.

\paragraph{Preservation and recovery of Aomoto realizations.}

The comparison immediately gives forward preservation.

\begin{introthm}[Forward preservation]
\label{thm:intro-forward-preservation}
Let $\nabla_A$ be an Aomoto realization of $\mathcal L$ with
respect to $\mathcal A$, satisfying $(\mathrm{SC}_\lambda)$ for
$\lambda\notin\mathbb Z$.
Then $\operatorname{mc}_\lambda^Y(\nabla_A)$ realizes
$\operatorname{MC}_\chi^{Y,\mathcal A}(\mathcal L)$ with respect
to $\mathcal B$, where $\chi(1)=e^{2\pi i\lambda}$.
If $\mathcal A$ is $Y$-closed, the realization remains on
$\mathcal A$.
\end{introthm}

In the inverse direction, we start from a chosen realization
of the transformed local system, without assuming a realization
of the original one.
Topological inversion recovers the local system on $M(\mathcal B)$.
Recovering an Aomoto realization on $\mathcal A$ additionally
requires removal of poles along $\mathcal B\setminus\mathcal A$.

For $D\in\mathcal B\setminus\mathcal A$, let $\mathscr C_D$
consist of the codimension-two flats $X$ with $X+Y=D$ that lie
in at least two members of $I$.
For residues $A'$ on $\mathcal B$, put
$A'_X:=\sum_{H\in\mathcal B,\,X\subset H}A'_H$ and consider
\[
\tag{$\mathrm{AP}$}
\left(\operatorname{Spec}(A'_D)\cup
    \bigcup_{X\in\mathscr C_D}\operatorname{Spec}(A'_X)\right)
    \cap(\mathbb Z\setminus\{0\})=\varnothing
    \qquad(D\in\mathcal B\setminus\mathcal A).
\]

\begin{introthm}[Inverse preservation]
\label{thm:intro-inverse-preservation}
Let $\mathcal L\in\operatorname{Loc}(M(\mathcal A))$ and $\chi\neq1$.
Assume that $\iota_Y^*\mathcal L$ satisfies Property $\wp_Y$
of Definition~\ref{df:fiberwiseP}
and the Kummer exclusion condition of Theorem~\ref{thm:composition}
for $\chi$.
Suppose that $\operatorname{MC}_\chi^{Y,\mathcal A}(\mathcal L)$
has an Aomoto realization
$\nabla_{A'}\in\operatorname{Pf}(\log\mathcal B)$ satisfying
$(\mathrm{SC}_\mu)$ for $\mu\notin\mathbb Z$ with
$e^{2\pi i\mu}=\chi(1)^{-1}$.

Then $\operatorname{mc}_\mu^Y(\nabla_{A'})$ realizes
$\iota_Y^*\mathcal L$ on $M(\mathcal B)$.
If $(\mathrm{AP})$ also holds, its residues along
$\mathcal B\setminus\mathcal A$ vanish, and the same connection
realizes $\mathcal L$ with respect to $\mathcal A$.
\end{introthm}

Property $\wp_Y$ is the local system version of the fiberwise
condition defining $\mathcal P^Y(\mathcal X)$.
Together with Kummer exclusion, it ensures that the perverse
inversion law descends to local systems without losing
boundary-supported objects.
These hypotheses hold, for example, if a fiber restriction
is irreducible with nontrivial monodromy at at least two
finite punctures.
For pole removal, the recovered local system has trivial
monodromy around each added hyperplane.
Our parallel residue estimate combines this with $(\mathrm{AP})$
to force the residue itself to vanish.
The pole removal condition is vacuous when $\mathcal A$
is $Y$-closed.

\paragraph{Further perspectives.}

The relative perverse construction suggests connections with the
multiplicative preprojective algebra description of middle
convolution \cite{CBS06} and the quiver description of
perverse sheaves on complexified real arrangements \cite{KS}.
The interpretation of algebraic middle convolution in terms of
Lie algebra homology \cite{Hiroe26} also points toward the framework
of Schechtman--Varchenko \cite{SV91}.
We plan to explore these connections in future work.

\paragraph{Organization of the paper.}

Section~\ref{sec:fibration} constructs the $Y$-closure and the
associated bundle of punctured affine lines.
Section~\ref{sec:topMC} develops relative middle convolution and
its composition laws.
Section~\ref{sec:additiveMC} describes algebraic convolution
and proves the parallel residue estimate.
Section~\ref{sec:compatibility} establishes the period comparison
and flatness, and Section~\ref{sec:Aomoto} applies them to
preservation and inverse recovery of Aomoto realizations.

\paragraph{AI Declaration.}

An initial version of this manuscript was prepared by the
authors without AI assistance.
During revision, Google Gemini and OpenAI ChatGPT were used
to assist with refining mathematical formulations,
checking and simplifying arguments, and improving exposition
and language.
All mathematical statements and arguments in the final
manuscript were independently verified by the authors,
who take full responsibility for its content.

\section{Fiber bundles associated with affine arrangements}
\label{sec:fibration}

We first characterize the directions for which projection is a bundle
of punctured affine lines, and then construct the minimal enlargement
of an arrangement for which this holds.

\subsection{The projection criterion}
\label{subsec:projection-criterion}

Let $\mathcal A$ be a finite affine hyperplane arrangement in
$V=\mathbb C^\ell$, and let
\[
    M(\mathcal A):=V\setminus\bigcup_{H\in\mathcal A}H
\]
be its complement. We write $L(\mathcal A)$ for the poset of nonempty
intersections of members of $\mathcal A$, including $V$, and
$L_2(\mathcal A)$ for its codimension-two flats.

Fix a line $Y\subset V$ through the origin and the quotient map
$p_Y:V\to V/Y$. An affine flat is parallel to $Y$ if its direction
contains $Y$. Set
\[
    \mathcal A_Y:=\{H\in\mathcal A:H\text{ is parallel to }Y\},
    \qquad I:=\mathcal A\setminus\mathcal A_Y,\qquad n:=|I|.
\]
The images of the members of $\mathcal A_Y$ form the arrangement
$p\mathcal A_Y:=\{p_Y(H):H\in\mathcal A_Y\}$ in $V/Y$.
Writing $S:=M(p\mathcal A_Y)$, we obtain a surjection
\[
    \pi_Y:M(\mathcal A)\longrightarrow S.
\]
For an affine flat $X$, the notation $X+Y$ means
$p_Y^{-1}(p_Y(X))$; in particular, $X+Y=X$ when $X$ is parallel to $Y$.

\begin{df}\label{df:good}
We say that $Y$ is \emph{good for $\mathcal A$}, or that
$\mathcal A$ is \emph{$Y$-closed}, if $X+Y\in L(\mathcal A)$ for every
$X\in L_2(\mathcal A)$.
\end{df}

The terminology ``$Y$-closed'' follows \cite{Oshi3}. For the relation
with fibration criteria for central arrangements, see \cite{Tera1}
and \cite[Lemma~3.20]{CDFSSTY}.

Choose a linear complement $V=W\oplus Y$ and a coordinate $y$ on $Y$.
Each $H\in I$ is uniquely a graph $y=a_H(\mathbf x)$, with $a_H$ affine
linear on $W$. Identifying $W$ with $V/Y$, we have
\begin{equation}\label{eq:fiber-description}
    M(\mathcal A)=\{(\mathbf x,y)\in S\times\mathbb C:
        y\neq a_H(\mathbf x)\text{ for all }H\in I\}.
\end{equation}
We use these coordinates throughout whenever a direction $Y$ is fixed.

\begin{prop}[cf.\ {\cite[Theorem~2.9]{Tera1}}]
\label{prop:TeraoAffine}
The following conditions are equivalent:
\begin{enumerate}
    \item\label{item:projection-closed} $\mathcal A$ is $Y$-closed;
    \item\label{item:projection-bundle}
        $\pi_Y:M(\mathcal A)\to S$ is a locally trivial topological
        fiber bundle;
    \item\label{item:projection-punctures}
        every fiber is an affine line with exactly $n$ points removed.
\end{enumerate}
\end{prop}

\begin{proof}
Condition~\eqref{item:projection-punctures} says that the values $a_H(\mathbf x)$ are pairwise
distinct for every $\mathbf x\in S$. For distinct $H,H'\in I$ with
nonempty intersection $X=H\cap H'$, their collision locus is the
hyperplane $p_Y(X)\subset V/Y$. It is disjoint from $S$ precisely when
it is a member of $p\mathcal A_Y$: a hyperplane contained in a finite
union of hyperplanes must equal one of them. Equivalently,
$X+Y\in\mathcal A_Y$. Disjoint $H,H'$ cause no collision.
For every other codimension-two flat, $X+Y$ is either $X$ itself or a
$Y$-parallel member of $\mathcal A$ containing $X$. Thus \eqref{item:projection-closed} and
\eqref{item:projection-punctures} are equivalent.

When these conditions hold, an ordering of $I$ gives a map
$\mathbf x\mapsto(a_H(\mathbf x))_{H\in I}$ from $S$ to
$\operatorname{Conf}_n(\mathbb C)$. By \eqref{eq:fiber-description},
$\pi_Y$ is the pullback of the Fadell--Neuwirth fibration
\[
    \operatorname{Conf}_{n+1}(\mathbb C)
        \longrightarrow\operatorname{Conf}_n(\mathbb C),
\]
which forgets the last point. This proves
\eqref{item:projection-punctures}$\Rightarrow$\eqref{item:projection-bundle}, with
$\operatorname{Conf}_0(\mathbb C)$ understood as a point.

Conversely, a generic fiber has $n$ punctures. Since $S$ is connected,
local triviality makes the first Betti number of the fibers constant.
An affine line with $r$ points removed has first Betti number $r$, so
every fiber has $n$ punctures, proving
\eqref{item:projection-bundle}$\Rightarrow$\eqref{item:projection-punctures}.
\end{proof}

\subsection{The $Y$-closure}
\label{subsec:Y-closure}

The preceding criterion can always be met by adjoining the
$Y$-parallel hyperplanes over the collision loci.

\begin{prop}\label{prop:Yclosure}
There is a unique smallest $Y$-closed arrangement containing
$\mathcal A$, called its \emph{$Y$-closure} and denoted by
$\mathcal B:=\overline{\mathcal A}^{\,Y}$. It is given by
\begin{equation}\label{eq:Y-closure}
    \mathcal B=\mathcal A\cup
    \{X+Y:X\in L_2(\mathcal A),\ X\text{ is not parallel to }Y\}.
\end{equation}
\end{prop}

\begin{proof}
All hyperplanes added in \eqref{eq:Y-closure} are parallel to $Y$.
For $Z=H_1\cap H_2\in L_2(\mathcal B)$, if both $H_i$ are parallel
to $Y$, then $Z+Y=Z$; if exactly one is parallel, then $Z+Y$ equals
that hyperplane. Otherwise both $H_i$ belong to $\mathcal A$, so
$Z+Y$ is already included in \eqref{eq:Y-closure}. This proves
$Y$-closedness. Any $Y$-closed enlargement of $\mathcal A$ must contain
$X+Y$ for every nonparallel $X\in L_2(\mathcal A)$, proving minimality
and uniqueness.
\end{proof}

Thus $I=\mathcal B\setminus\mathcal B_Y$, and the natural inclusion
\(
    \iota_Y:M(\mathcal B)\hookrightarrow M(\mathcal A)
\)
is obtained by deleting only $Y$-parallel hyperplanes. Proposition~\ref{prop:TeraoAffine} applies to $\mathcal B$.
For any $Y$-closed arrangement, we write
\[
    Q_n^{\mathbf x}:=\{a_H(\mathbf x):H\in I\},\qquad
    U_{\mathbf x}:=\pi_Y^{-1}(\mathbf x)
        \cong\mathbb C\setminus Q_n^{\mathbf x},
\]
and denote the fiber inclusion by
$\iota_{\mathbf x}:U_{\mathbf x}\hookrightarrow M(\mathcal A)$.
These punctured line bundles provide the correspondence used below.

\section{Topological middle convolution}
\label{sec:topMC}

We extend Katz's middle convolution to perverse sheaves constructible
for the relative stratification associated with a $Y$-closed arrangement.
The main result of this section is Theorem~\ref{thm:relative-composition}:
on the subcategory defined by fiberwise Property $\wp$, the relative
construction satisfies the composition and inversion laws.
Proposition~\ref{prop:relative-convolution-base-change} establishes the
required constructibility and compatibility with one-dimensional fibers.
We begin with the punctured line construction for local systems, then
prove the perverse sheaf result and determine when its composition law
passes to local systems.

All sheaves have complex coefficients and are considered in the complex
analytic topology. We write $\operatorname{Loc}(T)$ for finite-dimensional
local systems, $D_c^b(T)$ for the bounded constructible derived category,
and $\operatorname{Perv}(T)$ for its middle perverse heart; constructibility
is algebraic. The symbols $Rf_!$ and $Rf_*$ denote derived direct images,
whereas $R^if_!$ and $R^if_*$ denote their ordinary cohomology sheaves.
Tensor and external products of complexes are understood in the derived
sense.

For now, suppose $\mathcal A$ is $Y$-closed. Retain the notation
$S$, $I$, $n$, and $U_{\mathbf x}$ of Section~\ref{sec:fibration}, and
abbreviate $M:=M(\mathcal A)$.

\subsection{The punctured line correspondence}
\label{subsec:topMC-definition}

Let $\Delta_M$ be the relative diagonal in $M\times_S M$, and set
\[
    Z:=(M\times_S M)\setminus\Delta_M.
\]
In coordinates $(\mathbf x,t,y)$ on $Z$, define
\[
    p_1(\mathbf x,t,y)=(\mathbf x,t),\qquad
    p_2(\mathbf x,t,y)=(\mathbf x,y),\qquad
    \delta(\mathbf x,t,y)=y-t.
\]
By Proposition~\ref{prop:TeraoAffine}, $M\to S$ is a locally trivial
bundle of $n$-punctured affine lines. Treating $y$ as one additional
marked point gives the same Fadell--Neuwirth construction for $p_2$:
it is the pullback of
$\operatorname{Conf}_{n+2}(\mathbb C)\to
\operatorname{Conf}_{n+1}(\mathbb C)$ along
$(\mathbf x,y)\mapsto((a_H(\mathbf x))_{H\in I},y)$.
Thus $p_2$, and by symmetry $p_1$, is locally trivial with
$(n+1)$-punctured affine line fibers. The difference map $\delta$
takes values in $\mathbb C^\times$.

Let $\chi:\mathbb Z\to\mathbb C^\times$ be nontrivial, and let
$\mathcal K_\chi$ be the rank-one local system on $\mathbb C^\times$
with monodromy $\chi(1)$ around a positive meridian. For
$\mathcal L\in\operatorname{Loc}(M)$, set
\[
    \mathcal F_\chi(\mathcal L)
       :=p_1^*\mathcal L\otimes\delta^*\mathcal K_\chi,
\]
and put
\[
    C_{\chi,!}^Y(\mathcal L):=R^1p_{2!}\mathcal F_\chi(\mathcal L),
    \qquad
    C_{\chi,*}^Y(\mathcal L):=R^1p_{2*}\mathcal F_\chi(\mathcal L).
\]
These are local systems by local triviality of $p_2$.

\begin{df}\label{df:MC}
The \emph{middle convolution in the $Y$-direction} is the functor
$\operatorname{MC}_\chi^Y$ on $\operatorname{Loc}(M)$ defined by
\begin{equation}\label{eq:top-middle-convolution}
    \operatorname{MC}_\chi^Y(\mathcal L)
       :=\operatorname{Im}\bigl(C_{\chi,!}^Y(\mathcal L)
          \longrightarrow C_{\chi,*}^Y(\mathcal L)\bigr),
\end{equation}
where the arrow forgets supports and the image is taken in
$\operatorname{Loc}(M)$. We omit $Y$ when the direction is fixed.
\end{df}

For $m=(\mathbf x,y)\in M$, put
\[
    Q_{n+1}^m:=Q_n^{\mathbf x}\sqcup\{y\},\qquad
    U_m:=\mathbb C\setminus Q_{n+1}^m,
\]
and let
\[
    \mathcal F_m
       :=(\iota_{\mathbf x}^*\mathcal L)|_{U_m}
          \otimes(t\mapsto y-t)^*\mathcal K_\chi.
\]
Base change identifies the stalk with
\begin{equation}\label{eq:MC-stalk}
    \operatorname{MC}_\chi^Y(\mathcal L)_m
       \cong\operatorname{Im}\bigl(H_c^1(U_m,\mathcal F_m)
          \longrightarrow H^1(U_m,\mathcal F_m)\bigr).
\end{equation}

\begin{prop}\label{prop:katzreduction}
The functors $C_{\chi,!}^Y$ and $C_{\chi,*}^Y$ are exact. Moreover,
restriction to any fiber commutes with middle convolution:
\[
    \iota_{\mathbf x}^*\operatorname{MC}_\chi^Y(\mathcal L)
       \cong\operatorname{MC}_\chi^{U_{\mathbf x}}
          (\iota_{\mathbf x}^*\mathcal L),
\]
where the right-hand side denotes \eqref{eq:top-middle-convolution}
on the punctured affine line $U_{\mathbf x}$.
\end{prop}

\begin{proof}
The monodromy of $\mathcal F_m$ at the puncture $t=y$ is
$\chi(1)\operatorname{Id}$, so neither $\mathcal F_m$ nor its dual has
a nonzero global section. A punctured affine line has the homotopy
type of a finite graph. Poincar\'e duality therefore gives
\begin{equation}\label{eq:MC-vanishing}
    H^i(U_m,\mathcal F_m)=H_c^i(U_m,\mathcal F_m)=0
       \qquad(i\neq1).
\end{equation}
The cohomology long exact sequences imply exactness. Restriction of
the correspondence to $\mathbf x$ gives
$U_{\mathbf x}\times U_{\mathbf x}\setminus\Delta$ with the same
difference map. Base change for the locally trivial projections,
compatible with forgetting supports, gives the stated isomorphism.
\end{proof}

For later use, we also record the construction for an arbitrary
arrangement. Let $\mathcal B=\overline{\mathcal A}^{\,Y}$ and retain
$\iota_Y:M(\mathcal B)\hookrightarrow M(\mathcal A)$ from
Section~\ref{sec:fibration}.

\begin{df}\label{df:MC_general}
For $\mathcal L\in\operatorname{Loc}(M(\mathcal A))$ and $\chi\neq1$,
define
\[
    \operatorname{MC}_\chi^{Y,\mathcal A}(\mathcal L)
       :=\operatorname{MC}_\chi^{Y,\mathcal B}(\iota_Y^*\mathcal L)
       \in\operatorname{Loc}(M(\mathcal B)).
\]
The arrangement superscript is omitted when it is clear from the domain.
\end{df}

\subsection{Fiberwise perversity and Property $\wp$}
\label{subsec:relative-perverse-MC}

We now return to the $Y$-closed case and construct the relative
perverse sheaf functor underlying the local system construction.
Following \cite[\S2.5]{Katz}, we first recall the one-dimensional
derived convolutions: for $K,L\in D_c^b(\mathbb A^1)$, put
\[
    K*_!L:=Ra_!(K\boxtimes L),\qquad
    K*_*L:=Ra_*(K\boxtimes L),\qquad a(s,t)=s+t.
\]
No perverse truncation is implicit: even for perverse inputs, these
complexes need not be perverse.

Following \cite[(2.6.1)--(2.6.2)]{Katz}, a perverse sheaf $K$ has
\emph{Property $\wp$} if $L*_!K$ and $L*_*K$ are perverse for every
$L\in\operatorname{Perv}(\mathbb A^1)$. With such a $K$ fixed, the two
derived functors restrict to exact functors on
$\operatorname{Perv}(\mathbb A^1)$, and one defines
\[
    L*_{\mathrm{mid}}K
       :=\operatorname{Im}_{\operatorname{Perv}(\mathbb A^1)}
          \bigl(L*_!K\longrightarrow L*_*K\bigr).
\]
The image is thus taken in the perverse category, not in the derived
category. Over $\mathbb C$, Property $\wp$ is equivalent to
\begin{equation}\label{eq:Katz-P-Hom}
    \operatorname{Hom}(K,\mathbb C_{\mathbb A^1}[1])
       =\operatorname{Hom}(\mathbb C_{\mathbb A^1}[1],K)=0;
\end{equation}
see \cite[Lemmas~2.6.13--2.6.14 and Corollary~2.6.15]{Katz}.

For $k:\mathbb C^\times\hookrightarrow\mathbb A^1$, put
\[
    \mathscr K_\chi:=k_{!*}(\mathcal K_\chi[1])
       =(k_*\mathcal K_\chi)[1]\quad(\chi\neq1),\qquad
    \mathscr K_1:=\delta_0.
\]
Here $\delta_0$ is the skyscraper perverse sheaf with stalk $\mathbb C$
at the origin, the unit for additive middle convolution. The symbol
$k_*$ denotes ordinary sheaf direct image; for $\chi\neq1$,
$k_!\mathcal K_\chi\cong Rk_*\mathcal K_\chi$.
Each kernel has Property $\wp$. Middle convolution preserves this
property and is associative on the full subcategory of objects with
Property $\wp$. Moreover,
\begin{equation}\label{eq:Kummer-kernel-composition}
    \mathscr K_\chi*_{\mathrm{mid}}\mathscr K_\varphi
       \cong\mathscr K_{\chi\varphi}.
\end{equation}
These facts follow from
\cite[Lemma~2.6.5, Corollary~2.6.17, and Proposition~2.9.6]{Katz}.

\paragraph{The relative setting.}
Set $\mathcal X:=S\times\mathbb A^1$ and let
$j:M\hookrightarrow\mathcal X$ be the inclusion. Its complement is
the disjoint union of the sections
\[
    \Gamma_H
       :=\{(\mathbf x,y)\in\mathcal X:y=a_H(\mathbf x)\},
       \qquad H\in I.
\]
Let $\mathscr S$ be the stratification by $M$ and the $\Gamma_H$, and
write $\operatorname{Perv}_{\mathscr S}(\mathcal X)$ for its category
of constructible perverse sheaves. Put $b:=\dim S=\ell-1$.

For $\mathbf x\in S$, let $\mathscr S_{\mathbf x}$ be the stratification
of $\{\mathbf x\}\times\mathbb A^1\simeq\mathbb A^1$ by the punctured
line $U_{\mathbf x}$ and the points $a_H(\mathbf x)$, $H\in I$.
For $i_{\mathbf x}:\{\mathbf x\}\times\mathbb A^1\hookrightarrow\mathcal X$
and $F\in D_c^b(\mathcal X)$, write
\[
    F_{\mathbf x}:=i_{\mathbf x}^*F[-b]
\]
for normalized fiber restriction. The following criterion is the reason
that Katz's one-dimensional perversity statements can be applied in
this relative setting.

\begin{prop}[Fiberwise criterion]
\label{prop:fiberwise-perversity}
Let $F\in D_c^b(\mathcal X)$ be $\mathscr S$-constructible.
Then $F$ is perverse if and only if $F_{\mathbf x}$ is perverse
with respect to $\mathscr S_{\mathbf x}$ for every $\mathbf x\in S$.

Moreover, for every $\mathbf x\in S$, normalized fiber restriction
\[
    i_{\mathbf x}^*[-b]:
    \operatorname{Perv}_{\mathscr S}(\mathcal X)
       \longrightarrow
    \operatorname{Perv}_{\mathscr S_{\mathbf x}}(\mathbb A^1)
\]
is exact, and the family of these functors detects isomorphisms.
\end{prop}

\begin{proof}
Perversity is local on $\mathcal X$. Fix $(\mathbf x_0,y_0)\in\mathcal X$
and a sufficiently small contractible neighborhood $U\subset S$ of
$\mathbf x_0$. Away from the sections, the stratification is trivial.
At a point of $\Gamma_H$, the coordinate $z=y-a_H(\mathbf x)$ identifies
a neighborhood with $U\times\mathbb D$, stratified by
$U\times\mathbb D^\times$ and $U\times\{0\}$; no other section meets
this neighborhood because the $\Gamma_H$ are pairwise disjoint.

For either local product, pullback from the transverse disk identifies
its constructible derived category with the category for the product
stratification, since $U$ is contractible. With the fiber restricted
to this disk, this gives
\[
    F|_{U\times\mathbb D}
       \cong
    \operatorname{pr}_{\mathbb D}^*
       \bigl(F_{\mathbf x_0}|_{\mathbb D}\bigr)[b].
\]
The projection $\operatorname{pr}_{\mathbb D}$ is smooth and surjective
of relative complex dimension $b$. Hence
$\operatorname{pr}_{\mathbb D}^*[b]$ is perverse $t$-exact and detects
zero objects. Since it commutes with perverse cohomology, the two
objects are perverse simultaneously. This proves the first assertion.

For a short exact sequence
$0\to P'\to P\to P''\to0$ in
$\operatorname{Perv}_{\mathscr S}(\mathcal X)$, applying
$i_{\mathbf x}^*[-b]$ to the associated distinguished triangle gives
three perverse terms by the first assertion. The perverse cohomology
sequence therefore gives
\[
    0\longrightarrow P'_{\mathbf x}
      \longrightarrow P_{\mathbf x}
      \longrightarrow P''_{\mathbf x}
      \longrightarrow0.
\]
Hence normalized fiber restriction is exact. Finally, if a morphism
restricts to an isomorphism on every fiber, all stalks of its cone
vanish. The cone is therefore zero, proving the last assertion.
\end{proof}

Let $\mathcal P^Y(\mathcal X)$ be the full subcategory of
$\operatorname{Perv}_{\mathscr S}(\mathcal X)$ consisting of objects
$P$ for which every $P_{\mathbf x}$ has Property $\wp$.

\subsection{Relative convolution and base change}
\label{subsec:relative-convolution-base-change}
For $P\in D_c^b(\mathcal X)$ and $K\in D_c^b(\mathbb A^1)$, define
the relative derived convolutions by
\begin{equation}\label{eq:relative-convolution}
    P*_!^YK:=Ra_{S!}(P\boxtimes K),\qquad
    P*_*^YK:=Ra_{S*}(P\boxtimes K),
\end{equation}
where $a_S((\mathbf x,t),u)=(\mathbf x,t+u)$.
The next proposition shows that these operations preserve the prescribed
constructibility and commute with normalized fiber restriction.
The perverse functors will then be obtained without truncation.

\begin{prop}[Relative convolution]
\label{prop:relative-convolution-base-change}
Let $P\in\operatorname{Perv}_{\mathscr S}(\mathcal X)$ and let $K$
be perverse on $\mathbb A^1$, constructible with respect to
$\{\{0\},\mathbb C^\times\}$.
\begin{enumerate}
    \item The complexes in \eqref{eq:relative-convolution} are
    $\mathscr S$-constructible. For every $\mathbf x\in S$, there are
    natural isomorphisms in $D_c^b(\mathbb A^1)$,
    \[
        (P*_!^YK)_{\mathbf x}\cong P_{\mathbf x}*_!K,\qquad
        (P*_*^YK)_{\mathbf x}\cong P_{\mathbf x}*_*K,
    \]
    compatible with the maps which forget supports.

    \item Both convolutions are perverse if $K$ has Property $\wp$
    or if $P\in\mathcal P^Y(\mathcal X)$.
    For fixed $K$ with Property $\wp$, they are exact in $P$ on
    $\operatorname{Perv}_{\mathscr S}(\mathcal X)$.
    For fixed $P\in\mathcal P^Y(\mathcal X)$, they are exact in $K$
    among perverse sheaves constructible with respect to
    $\{\{0\},\mathbb C^\times\}$.
\end{enumerate}
\end{prop}

We prove the proposition by compactifying the convolution
correspondence. The two lemmas below establish base change at
infinity and constructibility along the prescribed target strata.

Fix $P$ and $K$ as in the proposition. Put
$\mathcal Y:=S\times\mathbb A^1_t\times\mathbb A^1_y$, and let
$p(\mathbf x,t,y)=(\mathbf x,y)$ and
$q(\mathbf x,t,y)=(\mathbf x,t)$.
The change of variables $y=t+u$ identifies the two convolutions with
$Rp_!\mathcal G$ and $Rp_*\mathcal G$, where
\[
    \mathcal G:=q^*P\otimes^L(y-t)^*K.
\]
Compactify the $t$-line by
$j_\infty:\mathcal Y\hookrightarrow
\overline{\mathcal Y}:=S\times\mathbb P^1_t\times\mathbb A^1_y$,
and let $\overline p:\overline{\mathcal Y}\to\mathcal X$ be the proper
projection. Then
\begin{equation}\label{eq:convolution-compactification}
    P*_!^YK\cong R\overline p_*j_{\infty!}\mathcal G,\qquad
    P*_*^YK\cong R\overline p_*Rj_{\infty*}\mathcal G.
\end{equation}
Denote the fiber inclusions over $S$ by
$\widetilde i_{\mathbf x}:\mathcal Y_{\mathbf x}\hookrightarrow\mathcal Y$
and
$\overline i_{\mathbf x}:\overline{\mathcal Y}_{\mathbf x}
\hookrightarrow\overline{\mathcal Y}$.
The induced open embedding is
$j_{\infty,\mathbf x}:\mathcal Y_{\mathbf x}
\hookrightarrow\overline{\mathcal Y}_{\mathbf x}$.

\begin{lem}[Base change at infinity]
\label{lem:convolution-at-infinity}
For every $\mathbf x\in S$, the canonical base change morphism
\[
    \overline i_{\mathbf x}^*Rj_{\infty*}\mathcal G
    \longrightarrow
    Rj_{\infty,\mathbf x*}\widetilde i_{\mathbf x}^*\mathcal G
\]
is an isomorphism. Moreover, the cohomology sheaves of
$Rj_{\infty*}\mathcal G$ restrict to local systems on
$S\times\{\infty\}\times\mathbb A^1_y$.
\end{lem}

\begin{proof}
Away from $t=\infty$, the base change morphism is an isomorphism
because $j_\infty$ is an open embedding. Near
$(\mathbf x_0,\infty,y_0)$, put $w=1/t$ and, after reordering the
factors, choose a product neighborhood $U\times V_0\times\mathbb D$.
Here $U\subset S$ and $V_0\subset\mathbb A^1_y$ are contractible,
and $\mathbb D$ is a disk centered at $w=0$; put
$\mathbb D^\times:=\mathbb D\setminus\{0\}$.

After shrinking, all functions $1-a_H(\mathbf x)w$ and $1-yw$ are
nonvanishing, since they equal $1$ at $w=0$. For $w\neq0$, the
identities $t-a_H(\mathbf x)=w^{-1}(1-a_H(\mathbf x)w)$ and
$t-y=w^{-1}(1-yw)$ show that $q$ takes values in $M$ and $y-t$ in
$\mathbb C^\times$ on the punctured product.
Since $P|_M$ and $K|_{\mathbb C^\times}$ are local systems shifted
by $b+1$ and $1$, respectively, $\mathcal G$ is a local system
shifted by $b+2$ there.

The projection
$U\times V_0\times\mathbb D^\times\to\mathbb D^\times$ is a homotopy
equivalence because $U\times V_0$ is contractible. Thus there is a
local system $\mathcal L_\infty$ on $\mathbb D^\times$ such that
\[
    \mathcal G\cong
    \mathbb C_{U\times V_0}\boxtimes\mathcal L_\infty[b+2].
\]
Locally, $j_\infty$ is
$\operatorname{id}_{U\times V_0}\times j_0$, where
$j_0:\mathbb D^\times\hookrightarrow\mathbb D$.
Computing the direct image on product neighborhoods gives
\[
    Rj_{\infty*}\mathcal G\cong
    \mathbb C_{U\times V_0}\boxtimes
       (Rj_{0*}\mathcal L_\infty)[b+2].
\]
On the fiber over $\mathbf x\in U$, both sides of the base change
morphism identify with
$\mathbb C_{V_0}\boxtimes(Rj_{0*}\mathcal L_\infty)[b+2]$.
The morphism restricts the constant $U$-factor to $\mathbf x$ and
acts as the identity on the remaining factor, so it is an isomorphism.
The same product description proves local constancy of the cohomology
sheaves along the boundary.
\end{proof}

\begin{lem}[Constructibility]
\label{lem:convolution-constructibility}
The complexes $R\overline p_*j_{\infty!}\mathcal G$ and
$R\overline p_*Rj_{\infty*}\mathcal G$ are
constructible with respect to $\mathscr S$.
\end{lem}

\begin{proof}
General constructibility follows from stability under extension by
zero, derived direct image along an open embedding, and proper direct
image. We check that no refinement of $\mathscr S$ is needed.

Fix a stratum $T$. Over $M$, the marked points
$a_H(\mathbf x)$, $y$, and $\infty$ are all distinct; over
$\Gamma_H$, only $y$ and $a_H(\mathbf x)$ coincide.
Merging this pair when necessary gives disjoint sections.
On a sufficiently small contractible open set $T_0\subset T$, an
isotopy identifies the family with a fixed marked sphere times
$T_0$.

Let
\(
    r:Z_0:=\overline p^{-1}(T_0)\longrightarrow T_0
\)
be the projection and $\mathcal E$ the restriction of either
$j_{\infty!}\mathcal G$ or $Rj_{\infty*}\mathcal G$ to $Z_0$.
Let $\Sigma$ be the union of the distinct marked sections, with
inclusions
\(
    h:Z_0\setminus\Sigma\hookrightarrow Z_0,
    \ 
    c:\Sigma\hookrightarrow Z_0.
\)

On $Z_0\setminus\Sigma$, the shifted local system $h^*\mathcal E$
is pulled back from a fiber: the punctured sphere family is trivial
and $T_0$ is contractible. Thus
$R(rh)_!h^*\mathcal E$ has locally constant cohomology.

On each finite marked section, $\mathcal E$ agrees with $\mathcal G$
because $j_\infty$ is an isomorphism there. The map $q$ takes values
in a single stratum of $\mathscr S$, while $y-t$ stays in either
$\{0\}$ or $\mathbb C^\times$. Constructibility of $P$ and $K$
therefore gives locally constant cohomology on that section.
At infinity, this follows from
Lemma~\ref{lem:convolution-at-infinity}, or from extension by zero.
Since $rc$ is a disjoint union of identity maps on $T_0$,
$R(rc)_*c^*\mathcal E$ also has locally constant cohomology.

As $r$ is proper, the localization triangle gives
\[
    R(rh)_!h^*\mathcal E
       \longrightarrow Rr_*\mathcal E
       \longrightarrow R(rc)_*c^*\mathcal E
       \xrightarrow{+1}.
\]
Both outer terms have locally constant cohomology, so the long exact
sequence gives the same for the middle term. By proper base change,
this is the restriction of the corresponding convolution to $T_0$,
proving $\mathscr S$-constructibility.
\end{proof}

\begin{proof}[Proof of Proposition~\ref{prop:relative-convolution-base-change}]
Constructibility follows from
Lemma~\ref{lem:convolution-constructibility} and
\eqref{eq:convolution-compactification}.

We next prove the fiber identities. In
\eqref{eq:convolution-compactification}, restriction to a fiber
commutes with extension by zero $j_{\infty!}$. For the other
extension, Lemma~\ref{lem:convolution-at-infinity} gives
\[
    \overline i_{\mathbf x}^*Rj_{\infty*}\mathcal G
       \cong
    Rj_{\infty,\mathbf x*}\widetilde i_{\mathbf x}^*\mathcal G.
\]
Proper base change for $\overline p$ then identifies the restrictions
of the relative convolutions with the corresponding direct images
on $\overline{\mathcal Y}_{\mathbf x}$.
On that fiber,
\[
    \widetilde i_{\mathbf x}^*\mathcal G[-b]
       \cong
    \operatorname{pr}_t^*P_{\mathbf x}\otimes^L(y-t)^*K,
\]
where $\operatorname{pr}_t:\mathbb A^1_t\times\mathbb A^1_y
\to\mathbb A^1_t$ is the projection. This is the coefficient complex
for the one-dimensional convolution, written in the coordinates
$y=t+u$. Taking its direct images therefore gives
\[
    (P*_!^YK)_{\mathbf x}\cong P_{\mathbf x}*_!K,\qquad
    (P*_*^YK)_{\mathbf x}\cong P_{\mathbf x}*_*K.
\]
These identifications are natural in $P$ and $K$ and respect
$j_{\infty!}\to Rj_{\infty*}$, hence the maps which forget supports;
compare \cite[(4.3.2)--(4.3.3)]{Katz}.

If $K$ has Property $\wp$, Katz's one-dimensional theory
\cite[(2.6.1)--(2.6.2)]{Katz} makes both fiber convolutions perverse.
If $P\in\mathcal P^Y(\mathcal X)$, the same follows from Property
$\wp$ for each $P_{\mathbf x}$ and commutativity of additive
convolution. Proposition~\ref{prop:fiberwise-perversity} and the
fiber identities therefore give perversity of both relative
convolutions.

For exactness, fix either factor as in the statement and take a short
exact sequence in the other perverse variable. Its convolution is
exact after normalized restriction to every fiber, by Katz's exactness
and the natural fiber identities. These restrictions are exact and
together detect zero objects by
Proposition~\ref{prop:fiberwise-perversity}, so the sequence is exact
in $\operatorname{Perv}_{\mathscr S}(\mathcal X)$ as well.
\end{proof}

\subsection{Relative middle convolution and composition}
\label{subsec:relative-middle-convolution}
For the inputs of Proposition~\ref{prop:relative-convolution-base-change},
whenever both derived convolutions are perverse, denote the image of
the forget-supports map in $\operatorname{Perv}_{\mathscr S}(\mathcal X)$
by $P*_{\mathrm{mid}}^YK$. In particular, for every
$P\in\operatorname{Perv}_{\mathscr S}(\mathcal X)$, define
\[
    \mathsf{MC}_\chi^Y(P):=P*_{\mathrm{mid}}^Y\mathscr K_\chi.
\]
For $\chi=1$, the kernel $\mathscr K_1=\delta_0$ gives the identity
functor. Proposition~\ref{prop:fiberwise-perversity} ensures that
normalized fiber restriction preserves perverse images. Applying it
to the compatible fiber isomorphisms of
Proposition~\ref{prop:relative-convolution-base-change} gives
\begin{equation}\label{eq:relative-MC-fiber}
    (\mathsf{MC}_\chi^Y(P))_{\mathbf x}
       \cong P_{\mathbf x}*_{\mathrm{mid}}\mathscr K_\chi.
\end{equation}

The following theorem establishes the composition law for the relative
perverse sheaf construction, before any restriction to local systems.

\begin{thm}[Relative composition law]\label{thm:relative-composition}
For every character $\chi$, the functor $\mathsf{MC}_\chi^Y$ preserves
$\mathcal P^Y(\mathcal X)$. On this category, for all characters
$\chi,\varphi$, there are natural isomorphisms
\[
    \mathsf{MC}_\varphi^Y \mathsf{MC}_\chi^Y
       \cong\mathsf{MC}_{\chi\varphi}^Y.
\]
In particular, $\mathsf{MC}_\chi^Y$ is an autoequivalence with inverse
$\mathsf{MC}_{\chi^{-1}}^Y$.
\end{thm}

\begin{proof}
Preservation follows from \eqref{eq:relative-MC-fiber} and
\cite[Corollary~2.6.17]{Katz}. For $P\in\mathcal P^Y(\mathcal X)$,
derived associativity identifies the two bracketings of the global
forget-supports map
\[
    (P*_!^Y\mathscr K_\chi)*_!^Y\mathscr K_\varphi
       \longrightarrow
    (P*_*^Y\mathscr K_\chi)*_*^Y\mathscr K_\varphi;
\]
see \cite[(2.5.3)(1)]{Katz}. Both terms are perverse by
Proposition~\ref{prop:relative-convolution-base-change}.
The same proposition with $S$ a point shows that the intermediate
kernel convolutions are perverse and constructible for
$\{\{0\},\mathbb C^\times\}$, so exactness in the kernel variable
also applies.

By the image factorization of \cite[Lemma~2.6.5]{Katz}, exactness
with the Kummer kernels fixed identifies the image of this map with
$\mathsf{MC}_\varphi^Y\mathsf{MC}_\chi^Y(P)$.
Exactness with $P$ fixed identifies the same image, using the other
bracketing, with
\[
    P*_{\mathrm{mid}}^Y
       (\mathscr K_\chi*_{\mathrm{mid}}\mathscr K_\varphi).
\]
The kernel identity \eqref{eq:Kummer-kernel-composition} gives the
composition law. Since $\mathscr K_1=\delta_0$ is the convolution
unit, $\mathsf{MC}_1^Y\cong\operatorname{Id}$.
Applying the composition law in both orders shows that
$\mathsf{MC}_\chi^Y$ and $\mathsf{MC}_{\chi^{-1}}^Y$
are mutually inverse.
\end{proof}

\subsection{Comparison on the open stratum}
\label{subsec:local-system-composition}

We now recover the local system construction from the relative
perverse sheaf functor. The comparison holds for every local system;
the composition law requires an additional condition excluding
boundary subobjects and quotients after the first convolution.

For $\mathcal L\in\operatorname{Loc}(M)$, put
$P_{\mathcal L}:=j_{!*}(\mathcal L[\ell])$ and
$\mathcal L_{\mathbf x}:=\iota_{\mathbf x}^*\mathcal L$.
In the local product models of
Proposition~\ref{prop:fiberwise-perversity}, middle extension commutes
with normalized fiber restriction. Hence, for
$j_{\mathbf x}:U_{\mathbf x}\hookrightarrow\mathbb A^1$,
\[
    (P_{\mathcal L})_{\mathbf x}
       \cong j_{\mathbf x,!*}(\mathcal L_{\mathbf x}[1])
       =(j_{\mathbf x,*}\mathcal L_{\mathbf x})[1],
\]
where the last direct image is the ordinary one, not
$Rj_{\mathbf x,*}$; see \cite[\S2.3.6]{Katz}.

We first compare the coefficient complexes for the two constructions.
Retain the projections $p$ and $q$ from
Section~\ref{subsec:relative-convolution-base-change}, put
$\mathcal Y_M:=p^{-1}(M)$, and let
$e:Z\hookrightarrow\mathcal Y_M$ be the natural inclusion.
Write $p_M:\mathcal Y_M\to M$ for the restriction of $p$, so that
$p_M\circ e=p_2$.

\begin{lem}[Comparison of coefficient complexes]
\label{lem:coefficient-comparison}
For $\mathcal L\in\operatorname{Loc}(M)$ and $\chi\neq1$, there are
natural isomorphisms on $\mathcal Y_M$,
\[
\begin{aligned}
    q^*(j_!\mathcal L[\ell])\otimes^L(y-t)^*\mathscr K_\chi
       &\cong
    e_!\mathcal F_\chi(\mathcal L)[\ell+1],\\
    q^*(Rj_*\mathcal L[\ell])\otimes^L(y-t)^*\mathscr K_\chi
       &\cong
    Re_*\mathcal F_\chi(\mathcal L)[\ell+1].
\end{aligned}
\]
These isomorphisms are compatible with the natural morphisms from
extension by zero to derived direct image.
\end{lem}

\begin{proof}
Both coefficient complexes restrict to
$\mathcal F_\chi(\mathcal L)[\ell+1]$ on $Z$. Adjunction gives
comparison morphisms from $e_!\mathcal F_\chi(\mathcal L)[\ell+1]$
to the first complex and from the second complex to
$Re_*\mathcal F_\chi(\mathcal L)[\ell+1]$.
It suffices to check them along the omitted sections
\[
    D_H:=\{t=a_H(\mathbf x)\}\quad(H\in I),\qquad
    D_y:=\{t=y\}.
\]
These are pairwise disjoint over $M$.

Near $D_H$, the coordinate $z=t-a_H(\mathbf x)$ identifies the
inclusions with products of a punctured disk inclusion and
contractible factors. The map $q$ forgets one of these factors,
so pullback by $q$ commutes with both extension by zero and
derived direct image in this local model. The Kummer factor
is a shifted local system across $D_H$, since $y-t\neq0$ there.
The projection formula therefore identifies the tensor products
with the corresponding extensions of their common restriction
$\mathcal F_\chi(\mathcal L)[\ell+1]$.

Near $D_y$, the map $q$ takes values in $M$, so the first factor
is a shifted local system across the section. The coordinate
$y-t$ is transverse to $D_y$, and $\chi(1)\neq1$ gives
$k_!\mathcal K_\chi\cong Rk_*\mathcal K_\chi$.
The projection formula now gives both comparisons at $D_y$.

Thus the adjunction morphisms are isomorphisms everywhere.
Their composite with the natural map between the coefficient
complexes restricts to the identity on $Z$; by adjunction it is
the canonical map from extension by zero to derived direct image.
This also proves compatibility and naturality.
\end{proof}

\begin{prop}\label{prop:perverse-local-comparison}
For every $\mathcal L\in\operatorname{Loc}(M)$ and $\chi\neq1$, there
is a natural isomorphism
\begin{equation}\label{eq:perverse-local-comparison}
    j^*\mathsf{MC}_\chi^Y(P_{\mathcal L})
       \cong \operatorname{MC}_\chi^Y(\mathcal L)[\ell].
\end{equation}
\end{prop}

\begin{proof}
The local product models near the $\Gamma_H$ identify the normalized
fiber restrictions of $j_!\mathcal L[\ell]$ and
$Rj_*\mathcal L[\ell]$ with the usual perverse extensions on a
punctured curve. Proposition~\ref{prop:fiberwise-perversity} therefore
shows that both complexes are perverse. Their natural map factors as
\[
    j_!\mathcal L[\ell]\twoheadrightarrow P_{\mathcal L}
       \hookrightarrow Rj_*\mathcal L[\ell].
\]
By Proposition~\ref{prop:relative-convolution-base-change}, convolution
with the fixed Kummer kernel is exact for both $*_!^Y$ and $*_*^Y$.
Naturality of the forget-supports map thus identifies
$\mathsf{MC}_\chi^Y(P_{\mathcal L})$ with the perverse image of
\[
    (j_!\mathcal L[\ell])*_!^Y\mathscr K_\chi
       \longrightarrow
    (Rj_*\mathcal L[\ell])*_*^Y\mathscr K_\chi.
\]

By Lemma~\ref{lem:coefficient-comparison}, open base change, and
$p_M\circ e=p_2$, the restrictions of these two terms to $M$ are
$(Rp_{2!}\mathcal F_\chi(\mathcal L))[\ell+1]$ and
$(Rp_{2*}\mathcal F_\chi(\mathcal L))[\ell+1]$, compatibly with
forgetting supports. The vanishing in \eqref{eq:MC-vanishing}
identifies them with $C_{\chi,!}^Y(\mathcal L)[\ell]$ and
$C_{\chi,*}^Y(\mathcal L)[\ell]$.
Since $j^*$ is exact on perverse sheaves, the restricted image is
therefore the image in $\operatorname{Loc}(M)$ shifted by $[\ell]$,
namely $\operatorname{MC}_\chi^Y(\mathcal L)[\ell]$.
This proves \eqref{eq:perverse-local-comparison}; compare
\cite[\S5.1.7]{Katz}.
\end{proof}

\subsection{Composition and inversion for local systems}
\label{subsec:local-system-inversion}
To apply Theorem~\ref{thm:relative-composition}, we first express its
fiberwise Property $\wp$ in terms of local systems.

\begin{df}\label{df:fiberwiseP}
A local system $\mathcal L$ on $M$ has \emph{Property $\wp_Y$} if
$P_{\mathcal L}\in\mathcal P^Y(\mathcal X)$. Equivalently,
\begin{equation}\label{eq:local-system-PY}
    \operatorname{Hom}(\mathcal L_{\mathbf x},
       \mathbb C_{U_{\mathbf x}})
       =
    \operatorname{Hom}(\mathbb C_{U_{\mathbf x}},
       \mathcal L_{\mathbf x})
       =0
       \qquad(\mathbf x\in S),
\end{equation}
where Hom is taken in $\operatorname{Loc}(U_{\mathbf x})$.
\end{df}

The equivalence follows from \eqref{eq:Katz-P-Hom} and full
faithfulness of middle extension, see \cite[\S2.3.3.1]{Katz}.
Since $S$ is connected and $\pi_Y$ is locally trivial, the condition
may be checked on a single fiber.

There is a separate obstruction to descending the composition law:
after the first convolution, boundary subobjects or quotients would
be lost upon restriction to $M$. The following translated Kummer
local systems describe this obstruction. For $H\in I$ and $\rho\neq1$,
let
\[
    \mathcal T_{\rho,H,\mathbf x}
       :=(t\mapsto t-a_H(\mathbf x))^*\mathcal K_\rho
       \quad\text{on }U_{\mathbf x}.
\]
This rank-one local system has monodromy $\rho(1)$ at
$a_H(\mathbf x)$ and trivial monodromy at the other finite punctures.
Its perverse middle extension is $\mathscr K_{\rho,a_H(\mathbf x)}$,
where $\mathscr K_{\rho,a}:=(t\mapsto t-a)^*\mathscr K_\rho$ denotes
the translated Kummer kernel.

\begin{thm}[Composition for local systems]
\label{thm:composition}
Let $\mathcal L\in\operatorname{Loc}(M)$ have Property $\wp_Y$ and
let $\chi\neq1$. Suppose that
\begin{equation}\label{eq:local-Kummer-exclusion}
    \operatorname{Hom}(\mathcal L_{\mathbf x},
       \mathcal T_{\chi^{-1},H,\mathbf x})
       =
    \operatorname{Hom}(\mathcal T_{\chi^{-1},H,\mathbf x},
       \mathcal L_{\mathbf x})
       =0
\end{equation}
for every $\mathbf x\in S$ and $H\in I$.
Then $\operatorname{MC}_\chi^Y(\mathcal L)$ has Property $\wp_Y$ and,
for every $\varphi\neq1$, there are natural isomorphisms
\begin{align}
    \operatorname{MC}_\varphi^Y
       \operatorname{MC}_\chi^Y(\mathcal L)
       &\cong
    \operatorname{MC}_{\chi\varphi}^Y(\mathcal L)
       &&(\chi\varphi\neq1),
       \label{eq:local-composition-law}\\
    \operatorname{MC}_{\chi^{-1}}^Y
       \operatorname{MC}_\chi^Y(\mathcal L)
       &\cong
    \mathcal L.
       \label{eq:local-inversion-law}
\end{align}
\end{thm}

\begin{proof}
Put $Q:=\mathsf{MC}_\chi^Y(P_{\mathcal L})$.
Theorem~\ref{thm:relative-composition} gives
$Q\in\mathcal P^Y(\mathcal X)$. On each fiber, the inverse equivalence
of \cite[Theorem~2.9.7]{Katz} sends
$\delta_{a_H(\mathbf x)}$ to $\mathscr K_{\chi^{-1},a_H(\mathbf x)}$.
Full faithfulness of this equivalence and of middle extension
identifies the Hom spaces in \eqref{eq:local-Kummer-exclusion} with
those from and to $\delta_{a_H(\mathbf x)}$ for $Q_{\mathbf x}$.
Their vanishing excludes punctual subobjects and quotients at every
removed point.

By Proposition~\ref{prop:fiberwise-perversity}, normalized fiber
restrictions are exact and together detect zero objects. Thus any
boundary subobject or quotient of $Q$ must vanish. The characterization
of middle extension, together with
Proposition~\ref{prop:perverse-local-comparison}, gives
\begin{equation}\label{eq:MC-recovered-middle-extension}
    Q\cong j_{!*}\bigl(\operatorname{MC}_\chi^Y(\mathcal L)[\ell]\bigr).
\end{equation}
Since $Q\in\mathcal P^Y(\mathcal X)$, this also proves Property
$\wp_Y$ for $\operatorname{MC}_\chi^Y(\mathcal L)$.
Finally, apply the relative composition law to $P_{\mathcal L}$ and
restrict to $M$. Equations \eqref{eq:perverse-local-comparison} and
\eqref{eq:MC-recovered-middle-extension} identify the result with
\eqref{eq:local-composition-law} and \eqref{eq:local-inversion-law}.
\end{proof}

Under Property $\wp_Y$, the proof also shows that
\eqref{eq:local-Kummer-exclusion} is equivalent to
\eqref{eq:MC-recovered-middle-extension}. The condition cannot be
omitted. On $\mathbb C^\times$, the local system
$\mathcal K_{\chi^{-1}}$ satisfies Property $\wp_Y$, but
\[
    \mathscr K_{\chi^{-1}}
       *_{\mathrm{mid}}\mathscr K_\chi
       \cong \delta_0.
\]
Its local system convolution is zero, so the local inversion formula
fails, although perverse middle convolution retains the punctual
object.

\begin{cor}\label{cor:irreducible-composition}
Suppose the restriction of $\mathcal L$ to a fiber is irreducible and
has nontrivial monodromy at at least two finite punctures. Then
\eqref{eq:local-composition-law} and
\eqref{eq:local-inversion-law} hold for all
$\chi,\varphi\neq1$, and
$\operatorname{MC}_\chi^Y(\mathcal L)$ has the same fiberwise
properties.
\end{cor}

\begin{proof}
An irreducible fiber local system with these properties is neither
constant nor a translated Kummer local system. Thus
\eqref{eq:local-system-PY} and
\eqref{eq:local-Kummer-exclusion} hold for every $\chi\neq1$.
Apply Theorem~\ref{thm:composition}; preservation of the fiberwise
properties follows from \cite[\S5.1.4]{Katz}.
\end{proof}

For an arbitrary arrangement, Definition~\ref{df:MC_general} applies
the same construction on the $Y$-closure $\mathcal B$. The preceding
composition and inversion laws therefore take the following form.

\begin{cor}\label{cor:compositiongeneral}
Suppose $\iota_Y^*\mathcal L$ satisfies the hypotheses of
Theorem~\ref{thm:composition} on $M(\mathcal B)$ for $\chi$.
For $\varphi\neq1$ and $\chi\varphi\neq1$,
\[
    \operatorname{MC}_\varphi^{Y,\mathcal B}
       \operatorname{MC}_\chi^{Y,\mathcal A}(\mathcal L)
       \cong
    \operatorname{MC}_{\chi\varphi}^{Y,\mathcal A}(\mathcal L).
\]
Moreover,
\begin{equation}\label{eq:general-inversion-law}
    \operatorname{MC}_{\chi^{-1}}^{Y,\mathcal B}
       \operatorname{MC}_\chi^{Y,\mathcal A}(\mathcal L)
       \cong
    \iota_Y^*\mathcal L.
\end{equation}
Ordinary direct image along $\iota_Y$ recovers $\mathcal L$.
\end{cor}

\begin{proof}
Apply Theorem~\ref{thm:composition} to $\mathcal B$.
For any local system $\mathcal V$ on $M(\mathcal A)$, the unit
\[
    \mathcal V
       \longrightarrow
    \iota_{Y*}\iota_Y^*\mathcal V
\]
is an isomorphism. Indeed, locally $\mathcal V$ is constant and the
complement of the additional complex hyperplanes is connected.
This proves the last assertion.
\end{proof}

\section{Algebraic middle convolution}
\label{sec:additiveMC}
We express Haraoka's convolution \cite{Har1} in terms of the
$Y$-closure, verify the invariant subspaces defining its middle
quotient, and establish the residue estimate needed to remove
apparent poles in Section~\ref{sec:Aomoto}.
We initially treat the convolution as an explicit matrix-valued
one-form: deducing its flatness from the period equations requires
the period map to have full rank.
Proposition~\ref{prop:fiber-period-isomorphism} establishes this
rank statement, and Proposition~\ref{prop:convolution-flatness}
then proves flatness for all parameters.

\subsection{Logarithmic Pfaffian systems}
\label{sec:logPfaff}

For each $H\in\mathcal A$, choose an affine linear defining equation
$f_H$. For a finite-dimensional complex vector space $E$ and
$A_H\in\operatorname{End}(E)$, set
\begin{equation}\label{eq:Pfaffian-form}
    \Omega_A:=\sum_{H\in\mathcal A}A_H\,d\log f_H,
    \qquad\nabla_A:=d-\Omega_A.
\end{equation}
Since $d\Omega_A=0$, flatness is equivalent to
$\Omega_A\wedge\Omega_A=0$. We write $\operatorname{Pf}(\log\mathcal A)$
for the category of these flat connections, with constant linear
morphisms $\phi:(E,A)\to(E',A')$ satisfying
$\phi A_H=A'_H\phi$ for every $H$.
Rescaling a defining equation by a nonzero constant does not change
\eqref{eq:Pfaffian-form}.
For any endomorphism $T$, $\operatorname{Spec}(T)$ denotes its set of
eigenvalues without multiplicities; no semisimplicity is assumed.

For $\nabla_A\in\operatorname{Pf}(\log\mathcal A)$, write
\[
    \mathcal L_A:=\operatorname{DR}_{\mathrm{Pf}}(\nabla_A)
       :=\ker\bigl(\nabla_A|_{M(\mathcal A)}\bigr).
\]
This is the unshifted horizontal section functor. It is exact on
flat holomorphic connections by
\cite[Theorem~I.2.17]{DelEDR}.

Fix $\nabla_A\in\operatorname{Pf}(\log\mathcal A)$ and retain
$\mathcal B=\overline{\mathcal A}^{\,Y}$, $I$, and $n$ from
Section~\ref{sec:fibration}. Set $A_D=0$ for
$D\in\mathcal B\setminus\mathcal A$, and normalize
$f_H=y-a_H(\mathbf x)$ for $H\in I$. The equations for the
$Y$-parallel hyperplanes are independent of $y$.
On the convolution space
\[
    \mathcal W:=E\otimes\mathbb C^I\cong E^I,\qquad
    \mathbb C^I:=\bigoplus_{H\in I}\mathbb C e_H,
\]
write $v=(v_H)_{H\in I}$ and put
\[
    \sigma_A(v):=\sum_{H\in I}A_Hv_H,\qquad
    A_I:=\sum_{H\in I}A_H.
\]
If $I$ is empty, the constructions give the zero system.

\subsection{The convolution and its residues}
\label{subsec:Haraoka-convolution}

For $\lambda\in\mathbb C$, define the matrix-valued one-form
\begin{equation}\label{eq:Haraoka-component-form}
\begin{aligned}
    (\Omega_{c_\lambda^Y(A)}v)_H
    &=\bigl(\sigma_A(v)+\lambda v_H\bigr)\,d\log f_H\\
    &\quad+\sum_{\substack{H'\in I\\H'\neq H}}
        A_{H'}(v_H-v_{H'})\,d\log(a_H-a_{H'})\\
    &\quad+\sum_{D\in\mathcal A_Y}A_Dv_H\,d\log f_D,
       \qquad H\in I.
\end{aligned}
\end{equation}
The difference $a_H-a_{H'}$ is not identically zero; its logarithmic
differential vanishes when it is constant. Write
\[
    c_\lambda^Y(\nabla_A):=d-\Omega_{c_\lambda^Y(A)}
\]
for the full convolution connection, whose flatness will be proved
in Proposition~\ref{prop:convolution-flatness}.

For a flat $X$ and a parallel hyperplane $D\in\mathcal B_Y$, put
\[
\begin{gathered}
    \mathcal F_X:=\{H\in I:X\subset H\},\qquad k_X:=|\mathcal F_X|,\\
    \mathscr C_D:=\{X\in L_2(\mathcal B):X+Y=D,\ k_X\ge2\}.
\end{gathered}
\]
Each $X\in\mathscr C_D$ is the intersection of two members of $I$,
so $L_2(\mathcal B)$ can equivalently be replaced by $L_2(\mathcal A)$
here. These collision classes are therefore unchanged on passing
to the $Y$-closure. For fixed $D$, the families $\mathcal F_X$ are
pairwise disjoint, since $H\in\mathcal F_X$ implies $X=H\cap D$.
Write
\[
    A_X:=\sum_{\substack{H\in\mathcal B\\X\subset H}}A_H.
\]
For $X\in\mathscr C_D$, the only parallel hyperplane containing $X$
is $D$, so $A_X=A_D+\sum_{H\in\mathcal F_X}A_H$.
We use this collision notation also for other coefficient tuples
on $\mathcal B$.

Let $E_{H,H'}$ be the matrix unit determined by
$E_{H,H'}e_J=\delta_{H',J}e_H$.

\begin{prop}\label{prop:Haraoka-identification}
The one-form $\Omega_{c_\lambda^Y(A)}$ agrees with Haraoka's
convolution formula in the $Y$-direction with parameter $\lambda$.
Its residue expansion
\[
    \Omega_{c_\lambda^Y(A)}
       =\sum_{D\in\mathcal B}C_D\,d\log f_D,
       \qquad C_D:=c_\lambda^Y(A)_D,
\]
has coefficients
\begin{equation}\label{eq:Haraoka-transverse-residue}
    C_D=\sum_{H\in I}
       (A_H+\lambda\delta_{D,H}\operatorname{Id}_E)\otimes E_{D,H},
       \qquad D\in I,
\end{equation}
and, for $D\in\mathcal B_Y$,
\begin{equation}\label{eq:Haraoka-parallel-residue}
    C_D=A_D\otimes\operatorname{Id}_{\mathbb C^I}
       +\sum_{X\in\mathscr C_D}\sum_{H'\in\mathcal F_X}A_{H'}\otimes
          \sum_{H\in\mathcal F_X}(E_{H,H}-E_{H,H'}).
\end{equation}
\end{prop}

\begin{proof}
Haraoka's selected coordinate $x$ and index set $\mathcal A_x$ are
our $y$ and $I$. The $dy$-coefficient of
\eqref{eq:Haraoka-component-form} is
\[
    \frac{\partial v_H}{\partial y}
       =\frac{\sigma_A(v)+\lambda v_H}{y-a_H(\mathbf x)},
\]
which is \cite[(15)]{Har1}.
For a base coordinate $\xi$, if $a_H-a_{H'}$ depends on $\xi$, write
$a_H-a_{H'}=\partial_\xi(a_H-a_{H'})(\xi-c_{HH'})$.
Its logarithmic derivative is $(\xi-c_{HH'})^{-1}$, and otherwise
is zero. Similarly, writing $f_H=(\partial_\xi f_H)(\xi-b_H)$ whenever
$f_H$ depends on $\xi$ identifies the $d\xi$-coefficient with
\cite[(16)--(17)]{Har1}. Thus the entire one-form is the convolution
of \cite[Definition~2.1]{Har1}, with the same parameter.

The nonconstant differences $a_H-a_{H'}$ vanish on
$(H\cap H')+Y$. These are precisely the possible added hyperplanes
in \eqref{eq:Y-closure}: a codimension-two flat contained in a parallel
member of $\mathcal A$ does not produce a new hyperplane.
Taking residues gives \eqref{eq:Haraoka-transverse-residue}, and
grouping collision terms by $X$ gives \eqref{eq:Haraoka-parallel-residue}.
\end{proof}

The residue formulas are independent of the auxiliary coordinates.
The arrangement $\mathcal B$ describes possible poles; some may be
absent for a given coefficient tuple because their residues vanish.

\subsection{The middle quotient}
\label{subsec:algebraic-middle-quotient}

Following \cite[\S2]{Har1}, define
\begin{equation}\label{eq:Haraoka-KL}
\begin{aligned}
    K&:=\bigoplus_{H\in I}(\ker A_H)\otimes\mathbb C e_H,\\
    L_\lambda&:=\ker\Bigl(\sum_{H\in I}C_H\Bigr)
       =\{v\in\mathcal W:\sigma_A(v)+\lambda v_H=0
          \text{ for every }H\in I\}.
\end{aligned}
\end{equation}
For $\lambda\neq0$, subtracting the defining equations shows
that every element of $L_\lambda$ is diagonal.
Thus, with $\Delta(u):=\sum_{H\in I}u\otimes e_H$, we have
\begin{equation}\label{eq:Haraoka-L-nonzero}
    L_\lambda=\Delta\bigl(\ker(A_I+\lambda\operatorname{Id}_E)\bigr),
    \qquad K\cap L_\lambda=0.
\end{equation}
Indeed, an element of the intersection is diagonal with $A_Hu=0$
for every $H$, and hence $\lambda u=0$.
For $\lambda=0$, we instead have
$L_0=\ker\sigma_A$ and $K\subset L_0$.

Both $K$ and $L_\lambda$ are invariant under every convolution
residue $C_D$, for all $\lambda\in\mathbb C$.
For the transverse residues, this follows directly from
\eqref{eq:Haraoka-KL}: on $K$,
$C_Hv=\lambda v_H\otimes e_H$, while $C_H$ annihilates
$L_\lambda$.
For the parallel residues, it follows by taking residues in
the invariance statement of \cite[Proposition~2.2]{Har1}.
That proposition includes $\lambda=0$, and its proof uses
only the integrability relations of the input system.
Consequently, the convolution one-form descends to
\[
    \overline{\mathcal W}_\lambda
       :=\mathcal W/(K+L_\lambda).
\]
Let $B_D$ denote the endomorphism of
$\overline{\mathcal W}_\lambda$ induced by $C_D$.

\begin{df}\label{df:algebraicMC}
Haraoka's \emph{algebraic middle convolution} is the connection
\begin{equation}\label{eq:algebraic-middle-convolution}
    \operatorname{mc}_\lambda^Y(\nabla_A)
       :=d-\sum_{D\in\mathcal B}B_D\,d\log f_D,
    \qquad B_D:=\operatorname{mc}_\lambda^Y(A)_D,
\end{equation}
on the trivial bundle with fiber $\overline{\mathcal W}_\lambda$;
see \cite[Definition~2.2]{Har1}.
\end{df}

Proposition~\ref{prop:convolution-flatness} will establish
flatness of the full convolution for every $\lambda$.
By the invariance above, flatness then passes to the middle
quotient.
A constant intertwiner $\phi$ between input systems induces
$\phi\otimes\operatorname{Id}_{\mathbb C^I}$ on their convolution
spaces, intertwining the residues $C_D$ and carrying $K$ and
$L_\lambda$ into the corresponding subspaces.
Thus, once flatness is established, both constructions
define functors
\[
    c_\lambda^Y,\ \operatorname{mc}_\lambda^Y:
       \operatorname{Pf}(\log\mathcal A)
       \longrightarrow\operatorname{Pf}(\log\mathcal B).
\]

\begin{rem}\label{rem:algebraic-MC-Yclosure}
Adjoining zero residues to regard $\nabla_A$ as an object of
$\operatorname{Pf}(\log\mathcal B)$ changes neither convolution:
$I$, the collision classes, and the subspaces
in \eqref{eq:Haraoka-KL} are unchanged.
\end{rem}

\subsection{Parallel residues and apparent poles}
\label{subsec:parallel-residues}

Fix $D\in\mathcal B_Y$ and use the collision notation of
Section~\ref{subsec:Haraoka-convolution}.

\begin{prop}\label{prop:parallel-spectrum}
For every $\lambda\in\mathbb C$,
\begin{equation}\label{eq:parallel-characteristic-polynomial}
\begin{aligned}
    \det(t\operatorname{Id}_{\mathcal W}-C_D)
    &=\det(t\operatorname{Id}_E-A_D)^{n-\sum_{X\in\mathscr C_D}(k_X-1)}\\
    &\quad\times\prod_{X\in\mathscr C_D}
       \det(t\operatorname{Id}_E-A_X)^{k_X-1}.
\end{aligned}
\end{equation}
The induced residue on the middle quotient satisfies
\begin{equation}\label{eq:parallel-spectrum-inclusion}
    \operatorname{Spec}(B_D)\subseteq\operatorname{Spec}(A_D)
       \cup\bigcup_{X\in\mathscr C_D}\operatorname{Spec}(A_X).
\end{equation}
\end{prop}

\begin{proof}
The disjoint families $\mathcal F_X$ give invariant blocks
$\mathcal W_X:=E^{\mathcal F_X}$. Let us write 
$\mathcal F_X=\{H_1,\ldots,H_k\}$. Then formula
\eqref{eq:Haraoka-parallel-residue} gives
\begin{equation}\label{eq:parallel-block-action}
    (C_Dv)_h=A_Dv_h+\sum_{j=1}^kA_{H_j}(v_h-v_j)
       =A_Xv_h-\sum_{j=1}^kA_{H_j}v_j.
\end{equation}
The diagonal subspace $\Delta_X(E)\subset E^k$ is invariant, with
action $A_D$. Modulo this subspace the last term in
\eqref{eq:parallel-block-action} vanishes, leaving $k-1$ copies of
$A_X$. Each remaining coordinate carries $A_D$. Multiplying the
characteristic polynomials of these invariant subspaces and quotients
gives \eqref{eq:parallel-characteristic-polynomial}.
An invariant quotient introduces no new eigenvalues, giving
\eqref{eq:parallel-spectrum-inclusion}.
\end{proof}

To use the spectral estimate, we record two local facts about
logarithmic connections. The first identifies the local cycles used
in Section~\ref{subsec:period-middle-quotient}; the second turns
trivial local monodromy into a criterion for a residue to vanish.

\begin{lem}\label{lem:local-logarithmic-connection}
Consider $g'(z)=(R/z+\Psi(z))g(z)$ on a punctured disk, with $R$
constant and $\Psi$ holomorphic on the disk.
\begin{enumerate}
    \item \label{item:local-logarithmic-spectral}
        If $\operatorname{Spec}(R)\cap(\mathbb Z\setminus\{0\})
        =\varnothing$, every single-valued horizontal section extends
        holomorphically to $0$, and evaluation at $0$ identifies the
        space of its germs with $\ker R$.

    \item\label{item:local-logarithmic-monodromy}
        If the local monodromy is the identity, then
        $\exp(2\pi iR)=\operatorname{Id}$.
        Thus $R$ is semisimple with integral eigenvalues;
        under the spectral condition in
        \eqref{item:local-logarithmic-spectral}, $R=0$.
\end{enumerate}
\end{lem}

\begin{proof}
For \eqref{item:local-logarithmic-spectral}, regular singularity implies that every single-valued solution
is meromorphic at $0$; see \cite[Theorem~11.13]{Forster}.
A nonzero leading term $z^kv$ satisfies $Rv=kv$, forcing $k=0$.
Evaluation is therefore injective with image in $\ker R$.
Conversely, for $v_0\in\ker R$ and
$\Psi(z)=\sum_{j\ge0}\Psi_jz^j$, the coefficients of
$g(z)=\sum_{k\ge0}v_kz^k$ are determined by
\[
    (k\operatorname{Id}-R)v_k
       =\sum_{j=0}^{k-1}\Psi_jv_{k-1-j},\qquad k\ge1.
\]
Invertibility and the estimate
$\|(k\operatorname{Id}-R)^{-1}\|=O(k^{-1})$ give convergence by the
usual majorant argument.

For \eqref{item:local-logarithmic-monodromy}, the Euler equation has monodromy $\exp(2\pi iR)$ by
\cite[Theorem~11.10]{Forster}. In the given trivialization, holonomy
along $z=\varepsilon e^{i\theta}$ converges to this matrix as
$\varepsilon\to0$, since the holomorphic term contributes
$O(\varepsilon)$ to the angular equation. Identity monodromy makes
every such holonomy the identity. The conclusions follow from the
Jordan normal form of $R$.
\end{proof}

\begin{cor}\label{cor:apparent-parallel-residue}
Suppose $D\in\mathcal B_Y$ satisfies
\[
\begin{aligned}
    \operatorname{Spec}(A_D)\cap(\mathbb Z\setminus\{0\})&=\varnothing,\\
    \operatorname{Spec}(A_X)\cap(\mathbb Z\setminus\{0\})&=\varnothing
       \qquad(X\in\mathscr C_D).
\end{aligned}
\]
If the local monodromy of $\operatorname{mc}_\lambda^Y(\nabla_A)$ at
a generic point of $D$ is trivial, then $B_D=0$. Hence $D$ can be
omitted from \eqref{eq:algebraic-middle-convolution} without a gauge
transformation.
\end{cor}

\begin{proof}
By \eqref{eq:parallel-spectrum-inclusion}, $B_D$ has no nonzero
integral eigenvalue. Restrict the connection to a transverse disk
at a point of $D$ on no other hyperplane of $\mathcal B$, and apply
Lemma~\ref{lem:local-logarithmic-connection}
\eqref{item:local-logarithmic-monodromy}.
\end{proof}

When $D\notin\mathcal A$, the condition on $A_D$ is automatic.
For inverse convolution, however, a hyperplane to be removed may
already carry a nonzero input residue. The criterion is therefore
stated for every $D\in\mathcal B_Y$.

\section{Compatibility of the de Rham functor and middle convolutions}
\label{sec:compatibility}

We compare topological and algebraic middle convolution through
fiberwise period integrals. After stating the comparison theorem,
we establish completeness of the periods by logarithmic de Rham
cohomology, deduce flatness of the convolution, and identify the
boundary cycles that define the middle quotient.

\subsection{The comparison theorem}
\label{sec:deRhamPfaff}

Retain $\mathcal B$, $I$, $A_I$, and $\mathcal L_A$ from
Section~\ref{sec:additiveMC}.
For the comparison, fix $\lambda\in\mathbb C\setminus\mathbb Z$ and
set $\chi(1)=\exp(2\pi i\lambda)$. Realize $\mathcal K_\chi$ as the
horizontal local system of $d-\lambda\,du/u$ on the trivial line
bundle over $\mathbb C^\times$.
The resulting Euler kernel is $(y-t)^\lambda$; thus our additive
parameter $\lambda$ corresponds to $\mu-1$ in
\cite[Definition~4.6 and Theorem~4.7(i)]{DR07}.

\begin{ass}\label{as:generic}
The coefficient matrices satisfy
\begin{equation}\label{eq:fiber-spectral-condition}
\begin{aligned}
    \operatorname{Spec}(A_H)\cap(\mathbb Z\setminus\{0\})&=\varnothing
       \qquad(H\in I),\\
    \operatorname{Spec}(A_I+\lambda\operatorname{Id}_E)
       \cap(\mathbb Z\setminus\{0\})&=\varnothing.
\end{aligned}
\end{equation}
\end{ass}

This is $(\mathrm{SC}_\lambda)$ from the Introduction.
It imposes no spectral restriction on the residues along
$Y$-parallel hyperplanes.

\begin{thm}[De Rham compatibility]\label{thm:MCdeRham}
Suppose that $\lambda\notin\mathbb Z$ and
Assumption~\ref{as:generic} holds. Then there is an isomorphism
of local systems on $M(\mathcal B)$,
\[
    \operatorname{MC}_\chi^{Y,\mathcal A}(\mathcal L_A)
       \cong\operatorname{DR}_{\mathrm{Pf}}
          \bigl(\operatorname{mc}_\lambda^Y(\nabla_A)\bigr),
\]
natural in $\nabla_A$.
\end{thm}

\begin{rem}[Earlier comparisons]
\label{rem:comparison-normalization}
Under the stated spectral conditions,
Theorem~\ref{thm:MCdeRham} requires neither irreducibility
nor the auxiliary monodromy conditions $(*)$ and $(**)$
used in \cite[Theorem~4.7 and Remark~4.8(a)]{DR07}.
It also allows integral differences between nonintegral
eigenvalues, dispensing with the nonresonance assumptions
in \cite[Theorem~5.5]{Har2}.

Two steps in the earlier integral proofs require further
justification: the nonvanishing argument used to establish
completeness in \cite[Lemmas~4.5 and~4.12]{DR07}, and the
identification of quotient representations in the proof of
\cite[Theorem~5.5]{Har2}.
Here Proposition~\ref{prop:fiber-period-isomorphism}
establishes completeness by logarithmic de Rham cohomology,
while Proposition~\ref{prop:boundary-cycle-identification}
identifies the period image of the boundary cycle subspace
with the subspace defining the algebraic middle quotient.
This identification is compatible with parallel transport
and supplies the compatibility needed to pass to the
middle quotient, which equality of dimensions alone
does not establish.
\end{rem}

By Remark~\ref{rem:algebraic-MC-Yclosure}, adjoining zero residues on
$\mathcal B\setminus\mathcal A$ changes neither the algebraic
convolution nor \eqref{eq:fiber-spectral-condition}. On local systems
it gives $\iota_Y^*\mathcal L_A$, as in Definition~\ref{df:MC_general}.
It therefore suffices to work on the $Y$-closure $\mathcal B$.
In the remainder of the section, put
\[
    M:=M(\mathcal B),\qquad S:=M(p\mathcal B_Y).
\]
Use the fiber spaces $U_m$ and twisted local systems $\mathcal F_m$
of Section~\ref{subsec:topMC-definition}, with the construction
applied to $\mathcal B$ and the input residues extended by zero.
We continue to write $\mathcal L_A$ for $\iota_Y^*\mathcal L_A$
and omit the arrangement superscript from convolution on $M$.

\subsection{Fiber periods and logarithmic cohomology}
\label{sec:period}

Fix $m=(\mathbf x,y)\in M$ and put $q_H:=a_H(\mathbf x)$.
The reduced boundary divisor of $U_m$ in $\mathbb P^1_t$ is
\[
    D_m:=\sum_{H\in I}[q_H]+[y]+[\infty].
\]
For $\mathcal L=\mathcal L_A$, the local system $\mathcal F_m$
is the horizontal local system of the tensor connection
\begin{equation}\label{eq:twisted-fiber-connection}
    \nabla_m^\lambda
       =d-\left(\sum_{H\in I}\frac{A_H}{t-q_H}
          +\frac{\lambda\operatorname{Id}_E}{t-y}\right)dt.
\end{equation}
Indeed, $d\log(y-t)=dt/(t-y)$ along the fiber.
The coefficient residues of the horizontal equation are
\begin{equation}\label{eq:twisted-fiber-residues}
    R_{q_H}=A_H,\qquad
    R_y=\lambda\operatorname{Id}_E,\qquad
    R_\infty=-A_I-\lambda\operatorname{Id}_E.
\end{equation}

Poincar\'e duality identifies
$C_{\chi,!}^Y(\mathcal L_A)_m$ with $H_1(U_m,\mathcal F_m)$.
For a twisted cycle $\gamma=\sum_\nu\sigma_\nu\otimes g_\nu$,
define
\[
    \operatorname{per}_{Y,m}(\gamma)
       :=\left(\sum_\nu\int_{\sigma_\nu}
          g_\nu(t)\,\frac{dt}{t-q_H}\right)_{H\in I}
       \in\mathcal W.
\]
Here each $g_\nu$ is horizontal for
\eqref{eq:twisted-fiber-connection} along $\sigma_\nu$.
This integral depends only on the homology class: for
$\eta\in E^\vee$, each scalar component pairs the cycle with
$dt/(t-q_H)\otimes\eta$, which is closed for the dual connection.

To deduce flatness of the convolution one-form from the period
equations, we need enough period vectors to form a local
fundamental matrix. The following proposition provides this
by identifying the fiber homology with the full space
$\mathcal W$ through the period map.

\begin{prop}[Completeness of periods]
\label{prop:fiber-period-isomorphism}
Under Assumption~\ref{as:generic}, the period map
\[
    \operatorname{per}_{Y,m}:H_1(U_m,\mathcal F_m)
       \xrightarrow{\ \sim\ }\mathcal W
\]
is an isomorphism for every $m\in M$, natural in the input
$\nabla_A$.
\end{prop}

\begin{proof}
The dual connection on the trivial bundle with fiber $E^\vee$ is
\[
    (\nabla_m^\lambda)^\vee
       =d+\left(\sum_{H\in I}\frac{A_H^\vee}{t-q_H}
           +\frac{\lambda\operatorname{Id}_{E^\vee}}{t-y}\right)dt,
    \qquad A_H^\vee\eta=\eta\circ A_H.
\]
Its residues as a connection operator are
$A_H^\vee$, $\lambda\operatorname{Id}_{E^\vee}$, and
$-(A_I+\lambda\operatorname{Id}_E)^\vee$ at $q_H$, $y$, and
$\infty$, respectively.
By Assumption~\ref{as:generic} and $\lambda\notin\mathbb Z$,
none has a positive integral eigenvalue. Deligne's logarithmic comparison
\cite[Corollary~II.6.10]{DelEDR} therefore identifies
$H^1(U_m,\mathcal F_m^\vee)$ with the first hypercohomology
of the complex
\[
    \mathcal O_{\mathbb P^1}\otimes E^\vee
       \xrightarrow{\ (\nabla_m^\lambda)^\vee\ }
       \Omega^1_{\mathbb P^1}(\log D_m)\otimes E^\vee.
\]

Since $\deg D_m=|I|+2$, we have
\[
    \Omega^1_{\mathbb P^1}(\log D_m)
       \cong\mathcal O_{\mathbb P^1}(|I|).
\]
Both terms of the complex thus have vanishing higher cohomology,
so its hypercohomology is computed by global sections.
In the logarithmic basis $dt/(t-q_H)$, $H\in I$, and $dt/(t-y)$,
the resulting differential is
\begin{equation}\label{eq:dual-logarithmic-differential}
    E^\vee\longrightarrow(E^\vee)^I\oplus E^\vee,\qquad
    \eta\longmapsto((A_H^\vee\eta)_{H\in I},\lambda\eta).
\end{equation}
As $\lambda\neq0$, each cokernel class has a unique representative
whose last component is zero. Hence
\[
\begin{aligned}
    (E^\vee)^I&\xrightarrow{\ \sim\ }H^1(U_m,\mathcal F_m^\vee),\\
    (\eta_H)_{H\in I}&\longmapsto
       \left[\sum_{H\in I}\frac{dt}{t-q_H}\otimes\eta_H\right].
\end{aligned}
\]

The integration pairing
\[
    H^1(U_m,\mathcal F_m^\vee)\times H_1(U_m,\mathcal F_m)
       \longrightarrow\mathbb C
\]
is perfect. Indeed, $U_m$ has the homotopy type of a finite graph,
and its cellular cochain complex with coefficients in
$\mathcal F_m^\vee$ is dual to the cellular chain complex
with coefficients in $\mathcal F_m$.
Dualizing the preceding isomorphism therefore gives
$\operatorname{per}_{Y,m}$.
Naturality follows because integration commutes with the constant
intertwiners defining morphisms of the input systems.
\end{proof}

We next recall the differential equations satisfied by these
periods as $m$ varies.

\begin{lem}[Period equations]
\label{lem:period-equations}
For a locally transported class in $H_1(U_m,\mathcal F_m)$,
the period vector $v=(v_H)_{H\in I}$ satisfies
\[
    dv=\Omega_{c_\lambda^Y(A)}v.
\]
\end{lem}

\begin{proof}
This is the differentiation calculation in
\cite[Proposition~2.1]{Har1}, expressed in our notation.
Locally, representatives of the transported classes may be chosen
with fixed paths avoiding the moving punctures, with coefficients
transported by the input connection and its Kummer twist.
Since these are closed twisted cycles, the boundary terms
in integration by parts cancel.
Haraoka's calculation therefore applies and gives
\eqref{eq:Haraoka-component-form}.
The change from Haraoka's kernel $(t-y)^\lambda$ to
$(y-t)^\lambda$ introduces only a constant factor after
compatible local branches are chosen.
The calculation does not require flatness of the output one-form.
\end{proof}

\subsection{Flatness and the full convolution comparison}
\label{subsec:convolution-flatness}

In the proof of \cite[Theorem~2.2]{Har1}, flatness is deduced
from the Euler integral representation of solutions.
This deduction requires enough linearly independent integral
solutions to form a local fundamental matrix, a point not
verified in that proof.
Proposition~\ref{prop:fiber-period-isomorphism} supplies this
rank statement under our spectral assumptions.
We combine it with Haraoka's period equations to establish
flatness under these assumptions, and then use the polynomial
dependence of the convolution residues to extend flatness
to every parameter and every flat input.

\begin{prop}[Flatness for all parameters]
\label{prop:convolution-flatness}
For every $\nabla_A\in\operatorname{Pf}(\log\mathcal A)$ and every
$\lambda\in\mathbb C$, the one-form
$\Omega_{c_\lambda^Y(A)}$ defines a flat connection on the trivial
bundle with fiber $\mathcal W$ over $M(\mathcal B)$.
The induced connection on $\mathcal W/(K+L_\lambda)$ is also flat.
\end{prop}

\begin{proof}
First suppose that $\lambda\notin\mathbb Z$ and
Assumption~\ref{as:generic} holds.
By Proposition~\ref{prop:fiber-period-isomorphism} and
Lemma~\ref{lem:period-equations}, the period vectors of a locally
transported homology basis form an invertible matrix $V$
satisfying $dV=\Omega_cV$, where
$\Omega_c=\Omega_{c_\lambda^Y(A)}$.
Thus
\[
    0=d^2V=(d\Omega_c-\Omega_c\wedge\Omega_c)V,
\]
which proves flatness under these assumptions.

For an arbitrary flat input $A$, scale all its residues,
including the parallel ones, by $\varepsilon$, keeping the
arrangement and the space $\mathcal W$ fixed.
Since $d\Omega_A=0$ and $\Omega_A\wedge\Omega_A=0$,
the input $\varepsilon A$ remains flat.
The spectral assumptions hold whenever $(\varepsilon,\lambda)$
is sufficiently close to $(0,1/2)$.
By the residue formulas in
Proposition~\ref{prop:Haraoka-identification}, the curvature
of $c_\lambda^Y(\varepsilon A)$ is polynomial in
$(\varepsilon,\lambda)$.
It vanishes on this nonempty open set by the preceding argument,
and hence vanishes identically.
Setting $\varepsilon=1$ proves flatness for every $A$ and $\lambda$.

The quotient connection is flat by the invariance established
in Section~\ref{subsec:algebraic-middle-quotient}.
\end{proof}

\begin{rem}\label{rem:flatness-versus-comparison}
The polynomial argument is applied to the full convolution
on the fixed space $\mathcal W$, before passage to the middle
quotient. Thus possible changes in $\dim(K+L_\lambda)$
do not affect the argument.
This extends flatness to all parameters; the period and middle
convolution comparisons still require
$\lambda\notin\mathbb Z$ and Assumption~\ref{as:generic}.
\end{rem}

We can now interpret the period maps as an isomorphism
of local systems.

\begin{thm}[Period isomorphism]\label{thm:period}
Suppose that $\lambda\notin\mathbb Z$ and
Assumption~\ref{as:generic} holds.
Local transport of period vectors gives an isomorphism
of local systems
\begin{equation}\label{eq:full-convolution-comparison}
    \operatorname{Per}_Y:C_{\chi,!}^Y(\mathcal L_A)
       \xrightarrow{\ \sim\ }\operatorname{DR}_{\mathrm{Pf}}
          \bigl(c_\lambda^Y(\nabla_A)\bigr).
\end{equation}
Under evaluation of horizontal germs at $m$, its stalk map
is $\operatorname{per}_{Y,m}$.
The isomorphism is natural in $\nabla_A$.
\end{thm}

\begin{proof}
By Proposition~\ref{prop:convolution-flatness}, the target
connection is flat.
Lemma~\ref{lem:period-equations} shows that locally transported
period vectors are horizontal in every variable.
Integration commutes with analytic continuation, so the local
period maps glue to a morphism of local systems.
Under evaluation at $m$, its stalk map is
$\operatorname{per}_{Y,m}$, which is an isomorphism by
Proposition~\ref{prop:fiber-period-isomorphism}.
Naturality follows because constant intertwiners commute
with integration.
\end{proof}

\subsection{Boundary cycles and the middle quotient}
\label{subsec:period-middle-quotient}

To pass from full to middle convolution, we describe the kernel
of the topological quotient map and compute its period image.

Fix $m\in M$. For a local system $\mathcal L$ on $M$, let
$\mathcal F_m$ denote its associated twisted local system on $U_m$
as in Section~\ref{subsec:topMC-definition}.
Choose disjoint small punctured disks $D_q^*$ at the support points
of $D_m$, and put
\[
    B_{\partial,m}:=
       \bigoplus_{q\in\operatorname{Supp}(D_m)}
          H_1(D_q^*,\mathcal F_m).
\]

\begin{prop}[Boundary quotient]\label{prop:MCcoim}
For any $\mathcal L\in\operatorname{Loc}(M)$ and nontrivial $\chi$,
there is a natural short exact sequence
\begin{equation}\label{eq:MC-homological-quotient}
    0\longrightarrow B_{\partial,m}
      \longrightarrow H_1(U_m,\mathcal F_m)
      \longrightarrow\operatorname{MC}_\chi^Y(\mathcal L)_m
      \longrightarrow0.
\end{equation}
The first map is induced by inclusion of the local cycles.
The sequence is compatible with parallel transport; shrinking
the disks gives canonically identified sequences.
\end{prop}

\begin{proof}
Delete small open disks at $\operatorname{Supp}(D_m)$ to obtain a
compact surface with boundary $T_m\subset U_m$.
Poincar\'e--Lefschetz duality identifies the map in \eqref{eq:MC-stalk}
with $H_1(T_m,\mathcal F_m)\to H_1(T_m,\partial T_m;\mathcal F_m)$.
By \eqref{eq:MC-stalk}, the middle convolution stalk is the image
of this map. The long exact sequence of the pair identifies its
kernel with the image of $H_1(\partial T_m,\mathcal F_m)$.
The latter maps injectively into $H_1(T_m,\mathcal F_m)$ because
\[
    H_2(T_m,\partial T_m;\mathcal F_m)
       \cong H^0(T_m,\mathcal F_m)=0
\]
by \eqref{eq:MC-vanishing}. Each boundary circle is homotopy equivalent
to $D_q^*$, giving \eqref{eq:MC-homological-quotient} and its naturality.
\end{proof}

We now specialize to $\mathcal L=\mathcal L_A$.
Via the injection in \eqref{eq:MC-homological-quotient}, regard
$B_{\partial,m}$ as a subspace of $H_1(U_m,\mathcal F_m)$.

\begin{prop}[Boundary periods]
\label{prop:boundary-cycle-identification}
Suppose that $\lambda\notin\mathbb Z$ and
Assumption~\ref{as:generic} holds. Then
\begin{equation}\label{eq:boundary-period-image}
    \operatorname{per}_{Y,m}(B_{\partial,m})=K\oplus L_\lambda.
\end{equation}
Identify local cycles with single-valued horizontal sections
by positive meridians, and then with their values at the punctures
by holomorphic extension. The local period maps are
\[
\begin{aligned}
    \ker A_H&\longrightarrow\mathcal W,
       &v&\longmapsto2\pi i\,v\otimes e_H,\\
    \ker(A_I+\lambda\operatorname{Id}_E)&\longrightarrow\mathcal W,
       &v&\longmapsto-2\pi i\,\Delta(v).
\end{aligned}
\]
At infinity, the positive meridian is taken in the coordinate
$w=1/t$. The moving puncture $t=y$ contributes zero.
The period map identifies the boundary cycle subsystem with
the horizontal local system of the convolution connection
restricted to $K\oplus L_\lambda$, naturally in $\nabla_A$.
\end{prop}

\begin{proof}
A positive local meridian identifies $H_1(D_q^*,\mathcal F_m)$ with
the space of single-valued horizontal sections.
Under the spectral assumptions,
Lemma~\ref{lem:local-logarithmic-connection}\eqref{item:local-logarithmic-spectral},
applied to \eqref{eq:twisted-fiber-residues}, identifies this space
with $\ker R_q$ by holomorphic extension and evaluation at $q$.
For $g(q_H)=v\in\ker A_H$, Cauchy's formula gives
\[
    \int_{\gamma_{q_H}}g(t)\,\frac{dt}{t-q_J}
       =2\pi i\,\delta_{H,J}v.
\]
Thus the cycles at the finite punctures $q_H$, $H\in I$,
have combined image $K$.
The moving puncture contributes nothing because its monodromy is
$\chi(1)\operatorname{Id}_E$ and $\chi(1)\neq1$.

At infinity, for the horizontal section $g_\infty$ with
$g_\infty(0)=v\in\ker(A_I+\lambda\operatorname{Id}_E)$, use
\[
    \frac{dt}{t-q_H}=-\frac{dw}{w(1-q_Hw)}.
\]
Integration around the positive meridian in $w$ gives $-2\pi i v$
in every component. The image is therefore
$\Delta\ker(A_I+\lambda\operatorname{Id}_E)=L_\lambda$ by
\eqref{eq:Haraoka-L-nonzero}.
Together with $K\cap L_\lambda=0$, this proves
\eqref{eq:boundary-period-image}.

By Proposition~\ref{prop:MCcoim}, the boundary cycles form a
local subsystem. Theorem~\ref{thm:period} identifies it with
a local subsystem of the horizontal sections of the convolution
connection. Its fiber image is exactly $K\oplus L_\lambda$,
so it is the horizontal local system of the restricted connection.
Naturality follows because integration and evaluation commute
with constant intertwiners.
\end{proof}

\begin{proof}[Proof of Theorem~\ref{thm:MCdeRham}]
Let $\mathscr N$ be the kernel of
$C_{\chi,!}^Y(\mathcal L_A)\to\operatorname{MC}_\chi^Y(\mathcal L_A)$.
Proposition~\ref{prop:MCcoim} identifies its stalk with
$B_{\partial,m}$. By Theorem~\ref{thm:period} and
Proposition~\ref{prop:boundary-cycle-identification}, its period image
is the horizontal local system of the restriction of
$\nabla_c:=d-\Omega_{c_\lambda^Y(A)}$ to $K\oplus L_\lambda$.
Exactness of horizontal sections for flat connections therefore gives
\[
\begin{aligned}
    \operatorname{MC}_\chi^Y(\mathcal L_A)
       &=C_{\chi,!}^Y(\mathcal L_A)/\mathscr N\\
       &\cong\operatorname{DR}_{\mathrm{Pf}}
          \bigl(\nabla_c/(K\oplus L_\lambda)\bigr)\\
       &=\operatorname{DR}_{\mathrm{Pf}}
          \bigl(\operatorname{mc}_\lambda^Y(\nabla_A)\bigr).
\end{aligned}
\]
The reduction in Section~\ref{sec:deRhamPfaff} gives the assertion
for arbitrary $\mathcal A$. Every map is natural in $\nabla_A$.
\end{proof}

\section{Preservation of Aomoto realizability}
\label{sec:Aomoto}

We apply the comparison theorem to realizability, retaining
$\mathcal B=\overline{\mathcal A}^{\,Y}$, $I$, and $\iota_Y$ from
Section~\ref{sec:fibration}. Forward preservation follows immediately.
In the inverse direction, the composition law recovers a local system
on $M(\mathcal B)$; the parallel residue estimate then removes the
additional poles.

\begin{df}\label{df:Aomoto-realization}
An \emph{Aomoto realization} of
$\mathcal L\in\operatorname{Loc}(M(\mathcal A))$ is an object
$\nabla_A\in\operatorname{Pf}(\log\mathcal A)$ together with an
isomorphism $\operatorname{DR}_{\mathrm{Pf}}(\nabla_A)\cong\mathcal L$.
We call $\mathcal L$ \emph{Aomoto-realizable with respect to $\mathcal A$}
if such an object exists.
\end{df}

The arrangement is part of the realization problem: a realization of
$\iota_Y^*\mathcal L$ on $\mathcal B$ does not by itself give a
realization of $\mathcal L$ on $\mathcal A$.

\subsection{Forward preservation}
\label{subsec:Aomoto-forward}

\begin{thm}[Forward preservation]\label{thm:Aomoto-forward}
Let $\nabla_A\in\operatorname{Pf}(\log\mathcal A)$ be an Aomoto
realization of $\mathcal L$. Suppose that it satisfies Assumption~\ref{as:generic} for $\lambda\in\mathbb C\setminus\mathbb Z$.
For $\chi(1)=\exp(2\pi i\lambda)$, the connection
$\operatorname{mc}_\lambda^Y(\nabla_A)$ realizes
$\operatorname{MC}_\chi^{Y,\mathcal A}(\mathcal L)$ with respect to
$\mathcal B$.
\end{thm}

\begin{proof}
By Proposition~\ref{prop:convolution-flatness}, the connection
belongs to $\operatorname{Pf}(\log\mathcal B)$.
Theorem~\ref{thm:MCdeRham} identifies
its horizontal local system with
$\operatorname{MC}_\chi^{Y,\mathcal A}(\mathcal L)$.
\end{proof}

Neither Property $\wp_Y$ nor Kummer exclusion is required here.
If $\mathcal A$ is $Y$-closed, the realization stays on $\mathcal A$.
The construction can be iterated, with varying directions and
parameters, whenever the chosen realization satisfies the spectral
condition at each step.

\subsection{Inverse preservation}
\label{subsec:Aomoto-inverse}

We now start from a realization of the transformed local system.
This realization need not arise by algebraic middle convolution
from a previously known realization of the original local system.
The topological inversion law recovers that local system on the
$Y$-closure; the residue criterion then gives sufficient conditions
for recovering a realization on the original arrangement.

Use $\mathscr C_D$, $\mathcal F_X$, and the residue sum notation from
Section~\ref{subsec:Haraoka-convolution}. These collision classes are
unchanged when the input arrangement is replaced by $\mathcal B$.
For an input tuple $A'$ on $\mathcal B$ and
$X\in\mathscr C_D$, in particular,
\[
    A'_X=A'_D+\sum_{H\in\mathcal F_X}A'_H.
\]
The term $A'_D$ need not vanish when $D\in\mathcal B\setminus\mathcal A$.

\begin{thm}[Inverse preservation]\label{thm:Aomoto-inverse}
Let $\chi\neq1$ and $\mathcal L\in\operatorname{Loc}(M(\mathcal A))$.
Suppose $\widetilde{\mathcal L}:=\iota_Y^*\mathcal L$ has Property
$\wp_Y$ and satisfies \eqref{eq:local-Kummer-exclusion} on
$M(\mathcal B)$ for $\chi$. Suppose
$\mathcal N:=\operatorname{MC}_\chi^{Y,\mathcal A}(\mathcal L)$ has
an Aomoto realization $\nabla_{A'}\in\operatorname{Pf}(\log\mathcal B)$
whose residues satisfy Assumption~\ref{as:generic}, with $A$
replaced by $A'$ and $\lambda$ by
$\mu\in\mathbb C\setminus\mathbb Z$, where
$\exp(2\pi i\mu)=\chi(1)^{-1}$.

Then
\begin{equation}\label{eq:inverse-Pfaffian-system}
    \nabla_R:=\operatorname{mc}_\mu^Y(\nabla_{A'})
       =d-\sum_{H\in\mathcal B}R_H\,d\log f_H
\end{equation}
realizes $\widetilde{\mathcal L}$ with respect to $\mathcal B$.
If moreover, for every $D\in\mathcal B\setminus\mathcal A$,
\begin{equation}\label{eq:inverse-added-spectra}
\begin{gathered}
    \operatorname{Spec}(A'_D)\cap(\mathbb Z\setminus\{0\})=\varnothing,\\
    \operatorname{Spec}(A'_X)\cap(\mathbb Z\setminus\{0\})=\varnothing
       \qquad(X\in\mathscr C_D),
\end{gathered}
\end{equation}
then $R_D=0$ for every such $D$. The same connection, written as
\[
    \nabla_{R,\mathcal A}
       :=d-\sum_{H\in\mathcal A}R_H\,d\log f_H
       \in\operatorname{Pf}(\log\mathcal A),
\]
realizes $\mathcal L$ on the original arrangement.
\end{thm}

Condition~\eqref{eq:inverse-added-spectra} is
$(\mathrm{AP})$ from the Introduction.

\begin{proof}
Apply Theorem~\ref{thm:MCdeRham} to $\nabla_{A'}$ and $\mu$, then use
\eqref{eq:general-inversion-law}. This gives
\begin{equation}\label{eq:inverse-realization-on-closure}
    \operatorname{DR}_{\mathrm{Pf}}(\nabla_R)
       \cong\operatorname{MC}_{\chi^{-1}}^{Y,\mathcal B}(\mathcal N)
       \cong\iota_Y^*\mathcal L.
\end{equation}
Now assume \eqref{eq:inverse-added-spectra} and fix
$D\in\mathcal B\setminus\mathcal A$. At a point of $D$ on no other
hyperplane of $\mathcal B$, a small transverse disk lies in
$M(\mathcal A)$ and its punctured disk lies in $M(\mathcal B)$.
Since $\mathcal L$ is defined on the entire disk, its meridian
monodromy is the identity. By \eqref{eq:inverse-realization-on-closure},
the same is true for $\nabla_R$.

Every added $D$ is parallel to $Y$. Applying Corollary~\ref{cor:apparent-parallel-residue} to $\nabla_{A'}$ and $\mu$ gives
$R_D=0$. Removing all these zero terms leaves the rational one-form
unchanged, hence defines the flat connection $\nabla_{R,\mathcal A}$.
Its horizontal local system and $\mathcal L$ have isomorphic
restrictions to $M(\mathcal B)$ by
\eqref{eq:inverse-realization-on-closure}. Applying ordinary
$\iota_{Y*}$ recovers both local systems, as proved in Corollary~\ref{cor:compositiongeneral}, so they are isomorphic on $M(\mathcal A)$.
\end{proof}

By Corollary~\ref{cor:irreducible-composition}, the topological
hypotheses hold if a fiber restriction of $\widetilde{\mathcal L}$ is
irreducible with nontrivial monodromy at at least two finite punctures.
If $\mathcal A$ is $Y$-closed, condition
\eqref{eq:inverse-added-spectra} is vacuous.

\begin{rem}\label{rem:inverse-spectral-hypotheses}
The spectral conditions concern the chosen realization $\nabla_{A'}$;
they do not follow from realizability alone.
Under \eqref{eq:inverse-realization-on-closure}, 
Lemma~\ref{lem:local-logarithmic-connection}
\eqref{item:local-logarithmic-monodromy} makes each additional $R_D$
semisimple with integral eigenvalues. Thus the exact condition for
removing its pole from the fixed presentation
\eqref{eq:inverse-Pfaffian-system} is
$\operatorname{Spec}(R_D)\cap(\mathbb Z\setminus\{0\})=\varnothing$.
Condition \eqref{eq:inverse-added-spectra} is sufficient by
Proposition~\ref{prop:parallel-spectrum}, but need not be necessary:
eigenvalues may disappear in the middle quotient.
\end{rem}

\bibliography{hre.bib}
\bibliographystyle{plain}
\end{document}